\documentclass[10pt]{article}
\usepackage[utf8]{inputenc}
\usepackage{amsfonts,amsmath,amssymb,amscd,amstext}
\usepackage{amsthm}
\newtheorem{theo}{Theorem}[section]
\newtheorem{lemma}{Lemma}[section]
\newtheorem{prop}{Proposition}[section]

\newtheorem{rem}{Remark}[section]
\newtheorem{cor}{Corollary}[section]
\newcommand{\R}{\mathbb{R}}
\newcommand{\C}{\mathbb{C}}
\newcommand{\N}{\mathbb{N}}

\title{The constraint equations in the presence of a scalar field - the case of the conformal method with volumetric drift}
\author{Caterina V\^alcu\thanks{Universit\'e Claude Bernard Lyon 1, 43 Boulevard du 11 Novembre 1918, 69100 Villeurbanne, \texttt{valcu@math.univ-lyon1.fr}}}
\begin{document}
\maketitle
\begin{abstract}
In this paper we establish the existence in low dimensions of solutions to the constraint equations in the case of the conformal system recently proposed by David Maxwell \cite{Max14}, with the added presence of a scalar field and under suitable smallness assumptions on its parameters.
\end{abstract}

\section{Introduction}
The field of general relativity deals with the study of spacetime, an object defined as the equivalence class, up to an isometry, of Lorentzian manifolds $(\tilde{M},\tilde{g})$ of dimension $n+1$ satisfying the Einstein field equations
$$Ric_{\alpha\beta}(\tilde{g})-\frac{1}{2}R(\tilde{g})\tilde{g}_{\alpha\beta}=8\pi T_{\alpha\beta},$$
$\alpha,\beta=\overline{1,n+1}$. Here, $R(\tilde{g})$ is the scalar curvature of $\tilde{g}$, $Ric$ the Ricci curvature and $T_{\alpha\beta}$ the stress-energy tensor describing the presence of matter and energy. For example, $T_{\alpha\beta}=0$ describes vacuum. Our interest focuses on the more general case
$$T_{\alpha\beta}=\nabla_\alpha\tilde\psi\nabla_\beta\tilde\psi-\left(\frac{1}{2}|\nabla\tilde\psi|^2_{\tilde{g}}+V(\tilde\psi)\right)\tilde{g}_{\alpha\beta},$$
which models the existence within the spacetime of a scalar field $\tilde\psi\in\mathcal{C}^\infty(M)$ having potential $V\in\mathcal{C}^\infty(\R)$. Thus, $\tilde\psi=0$ and $V=\Lambda$ yield the vacuum with cosmological constant $\Lambda$, while $V=\frac{1}{2}m\tilde\psi^2$ corresponds to the Einstein-Klein-Gordon setting.
\par For a globally hyperbolic spacetime, we define its initial data $(M,\hat{g},\hat{K},\hat{\psi},\hat{\pi})$. They consist of an $n$-dimensional Riemannian manifold $(M,\hat{g})$, which models the spacetime at a particular moment in time, a symmetric 2-tensor $\hat{K}$, corresponding to its second fundamental form, the scalar field $\hat{\psi}$ in $M$, and its temporal derivative $\hat{\pi}$. The associated spacetime development takes the form $(M\times \R,\tilde{g}, \tilde{\psi})$, where $\tilde{g}$ is a Lorentzian metric that verifies $\tilde{g}|_{M}=\hat{g}$ and $\tilde\psi$ is a scalar field such that $\tilde{\psi}|_{M}=\hat{\psi}$ and $\partial_t\tilde{\psi}|_{M}=\hat{\pi}$.
\par Initial data in general relativity may not be freely specified, unlike their Newtonian counterparts. Instead, they must verify the Gauss and Codazzi equations,
\begin{equation*}
\begin{cases}
\displaystyle R(\hat{g})+(tr_{\hat{g}}\hat{K})^2-||\hat{K}||^2_{\hat{g}} &\displaystyle= \hat{\pi}^2+|\hat{\nabla}\hat{\psi}|^2_g+2V(\hat{\psi})\\
\partial_i(tr_{\hat{g}} \hat{K})-\hat{K}_{i,j}^j&\displaystyle= \hat{\pi}\partial_i\hat{\psi},
\end{cases}
\end{equation*}
which are referred to as the \textit{constraint equations}. The work of Choquet-Bruhat \cite{Cho} establishes, once and for all, that the constraint equations are not only necessary but sufficient conditions for the (local) existence of a solution. Later, Choquet-Bruhat and Geroch \cite{ChoGer} prove that the maximal development of initial data is unique, up to an isometry. Globally hyperbolic spacetimes may rigorously be studied in the context of mathematical analysis as the result of an evolution problem. The above system is clearly under-determined, which allows for considerable freedom in choosing a solution $(\hat{g},\hat{K},\hat{\psi},\hat{\pi})$.
\par Using the conformal method introduced by Lichnerowicz \cite{Lic}, the constraint equations may be transformed into a determined system of equations by fixing well-chosen quantities (see Choquet-Bruhat, Isenberg and Pollack \cite{ChoIsePol}). The appeal of such a method lies in that it provides a characterisation of the resulting initial data by fixed quantities. Essentially, it maps a space of parameters to the space of solutions.
\par Given an initial data set $(\hat{g},\hat{K},\hat{\psi},\hat{\pi})$, the classical choice of parameters is $(\mathbf{g},\mathbf{U},\tau,\psi,\pi;\alpha)$: in this case, the conformal class $\mathbf{g}$ is represented by a Riemannian metric $g$, the smooth function  $\tau=\hat{g}^{ab}\hat{K}_{ab}$ is a mean curvature and the conformal momentum $\mathbf{U}$ measured by a volume form $\alpha$ (volume gauge) is a 2-tensor that is both trace-free and divergence-free with respect to $g$ (a transverse-traceless tensor). We sometimes prefer to indicate the volume gauge by the densitized lapse 
$$\tilde{N}_{g,\alpha}:=\frac{\alpha}{dV_g}.$$
Note that this quantity depends on the choice of representative $g$, unlike the volume gauge $\alpha$ which does not.
The standard conformal method implicitly fixes $\tilde{N}_{g,\alpha}=2$; in the present paper, we prefer to make use of the freedom of choosing $\tilde{N}_{g,\alpha}$ as needed. We often refer to a parameter set by indicating the representative metric $g$ and the corresponding densitized lapse $\tilde{N}_{g,\alpha}$ instead of giving the conformal class and volume gauge. However, these quantities can immediately be reconstructed from our data. We refer to Maxwell \cite{Max14} for an introduction to the conformal method in our context.
\par Starting from the parameter set $(g,U,\tau,\psi,\pi;\tilde{N})$, the corresponding (physical) initial data is pinpointed by solving a resulting system, comprising the Lichnerowicz-type equation and the momentum constraints, for a smooth positive function (or conformal factor) $u$ and a smooth vector field $W$ in $M$, 
\begin{equation}
\begin{cases}\label{syst classical conformal method}
\displaystyle \Delta_g u+\mathcal{R}_\psi=&\displaystyle-\mathcal{B}_{\tau,\psi,V}u^{q-1}+\frac{\mathcal{A}_{\pi,U}(W)}{u^{q+1}},\\ 
\Delta_{g, conf}W=&\displaystyle \frac{n-1}{n}u^{q}\nabla\tau+\pi\nabla\psi,
\end{cases}
\end{equation}
where
\begin{equation*}
\begin{array}{l}
\displaystyle \mathcal{R}_\psi=\frac{n-2}{4(n-1)} \left(R(g)-|\nabla\psi|^2_g\right),\\  
\displaystyle \mathcal{B}_{\tau,\psi,V}=\frac{n-2}{4(n-1)}\left(\frac{n-1}{n}\tau^2-2V(\psi)\right),\\ 
\displaystyle \mathcal{A}_{\pi,U}(W)=\frac{n-2}{4(n-1)}\left(|U+\mathcal{L}_gW|^2_g+\pi^2\right).
\end{array}
\end{equation*}
If $(u,W)$ solves the above system, then the initial data we've been searching for are
$$\hat{g}=u^{q-2}g,\quad \hat{K}=u^{-2}\left(U+\frac{\tilde{N}}{2}\mathcal{L}_g W\right)+\frac{\tau}{n}\hat{g},\quad \hat{\psi}=\psi,\quad\hat{\pi}=u^{-q}\pi.$$
Note also that the solutions generated by $(g,U,\tau;\tilde{N})$ and $$(\varphi^{q-2}g,\varphi^{-2}U,\tau,\psi,\varphi^{-q}\pi; \varphi^q \tilde{N})$$ are the same, where $\varphi$ is a smooth positive function. The notations above are similar to those of Choquet-Bruhat, Isenberg and Pollack (\cite{ChoIsePol}, \cite{ChoIsePol2}). The following quantities often appear throughout the present paper:  $q=\frac{2n}{n-2}$ is the critical Sobolev exponent for the embedding of $H^1$ in Lebesgue spaces, $\Delta_g=-div_g\nabla$ denotes the Laplace-Beltrami operator taken with non-negative eigenvalues, $\Delta_{g,conf}W=div_g(\mathcal{L}_g W)$ is the Lam\'e operator and $\mathcal{L}_g$ is the conformal Killing operator with respect to $g$,
$$\mathcal{L}_g W_{ij}=W_{i,j}+W_{j,i}-\frac{2}{n}div_g W g_{ij}.$$
Conformal Killing fields are defined as vector fields in the kernel of $\mathcal{L}_g.$
\par The conformal method is particularly successful in finding solutions when the mean curvature $\tau$ is constant as the system (\ref{syst classical conformal method}) becomes uncoupled, but it is unclear how well the method functions when the mean curvature is far from being constant: see Maxwell \cite{Max11} and \cite{Max14c}, where a given set of parameters point to no or to an infinite number of solutions. We emphasize that any failing of the system does not necessarily translate to a singularity in the space of solutions to the constraints system, but may instead derive from a poor choice of mapping. This motivates the study of variations to standard conformal methods.
\par The drift method introduced by Maxwell replaces the mean curvature $\tau$ as a parameter by a pair $(\tau^*,\tilde{V})$, where $\tau^*$ is a unique constant called volumetric momentum and $\tilde{V}$ a vector field related to the drift. They verify an analogue of York splitting, namely
\begin{equation}\label{volumetric momentum}
\tau=\tau^*+\tilde{N}_{\hat{g},\alpha}div_{\hat{g}}\tilde{V}=\tau^*+\frac{\tilde{N}_{g,\alpha}}{u^{2q}}div_g(u^q\tilde{V}),
\end{equation}
the notation $\tilde{V}$ being specific to this paper in order to avoid confusion with the potential $V$. Interestingly, $\tau^*=0$ holds true for all counterexamples found by Maxwell (\cite{Max11}, \cite{Max14c}). This suggests that the volumetric momentum may play an important role in characterizing the space of initial data. Ideally, we would like to know as soon as we fix a set of parameters $(g,U,\tau,\psi,\pi;\tilde{N})$ if we find ourselves in the case $\tau^*=0$. However, $\tau^*$ cannot be directly calculated by (\ref{volumetric momentum}) without first solving (\ref{syst classical conformal method}), which somewhat defeats the purpose. This motivates a new choice of parameters, even at the risk of working with an analytically more complicated system. The idea of Maxwell (\cite{Max11}, \cite{Max14c}) is thus to choose $\tau^*$ as an additional parameter that is to be fixed in the place of $\tau$: the hope is thus to avoid the aforementioned problem. As well  as $\tau^*$, Maxwell added $\tilde{N}$ and $\tilde{V}$ as parameters, for geometric and physical reasons. Therefore, instead of fixing $\tau$, we fix $\tau^*$, $\tilde{N}$ and $\tilde{V}$.
\par Intuitively, the drift is a geometric quantity describing infinitesimal motion in the space of metrics modulo the group of diffeomorphisms connected to the identity such that the conformal class and volume are preserved. For any given drift, the choice of a representative vector field $\tilde{V}$ is unique up to conformal Killing fields and vector fields which are divergence-free with respect to the initial metric $\hat{g}$. Given $g$ an arbitrary representative of the conformal class, it is not clear whether two vector fields are indicative of the same drift class defined for $\hat{g}$; this problem is discussed at length in the paper of Mike Holst, David Maxwell and Rafe Mazzeo \cite{HolMaxMaz} for conformal systems where the critical non-linearity is non-focusing, or negative. Our analysis treats systems with focusing (that is to say positive) non-linearities stemming from the presence of a scalar field with positive potential. 
\par The following system corresponds to Problem 12.1 of \cite{Max14} in the presence of a scalar field, where $g$ admits no non-trivial conformal Killing fields:
\begin{equation}\label{syst of Maxwell}
	\begin{cases}
	\Delta_g u+\frac{n-2}{4(n-1)}(R(g)-|\nabla\psi|_g^2)u -\frac{(n-2)|U+\mathcal{L}_g W|^2+\pi^2}{4(n-1)u^{q+1}}\\ 
	\displaystyle\quad-\frac{n-2}{4(n-1)}\left[2V(\psi)-\frac{n-1}{n}\left(\tau^*+\frac{\tilde{N}div_g(u^q\tilde{V})}{u^{2q}}\right)^2\right]u^{q-1}=0\\ \\
	 div_g\left(\frac{\tilde{N}}{2}\mathcal{L}_g W\right)-\frac{n-1}{n}u^q \mathbf{d}	\left(\frac{\tilde{N}div_g(u^{q}\tilde{V})}{2u^{2q}}\right)-\pi\nabla\psi=0.
	\end{cases}
	\end{equation}
We denote the exterior derivative by $\mathbf{d}$. The unknowns are a smooth positive scalar function $u$ defined on $M$ and a smooth vector field $W$ on $M$. The parameters are $(g,U,\tau^*,\tilde{V},\psi,\pi;\tilde{N})$. Maxwell's new set of parameters include $\tau^*$, which could not be calculated \textit{a priori} in the classical method. The initial data of the constraint equations verify
\begin{equation}
\begin{array}{c}
\hat{g}=u^{q-2}g,\quad \hat{K}=u^{-2}\left(U+\frac{\tilde{N}}{2}\mathcal{L}_g W\right)+\frac{1}{n}\left(\tau^*+\frac{\tilde{N}}{u^{2q}}div(u^q\tilde{V})\right)\hat{g},\\ \hat{\psi}=\psi,\quad\hat{\pi}=u^{-q}\pi.
\end{array}
\end{equation}
\par The following is a more general system than (\ref{syst of Maxwell}). The central result of the paper consists in showing that it admits solutions.  Let $(M,g)$ be a closed Riemannian manifold of dimension $n\in\{3,4,5\}$, and $g$ has no non-trivial conformal Killing fields. Let $b$, $c$, $d$, $f$, $h$, $\rho_1$, $\rho_2$, $\rho_3$ be smooth functions on $M$ and let $Y$ and $\Psi$ be smooth vector fields defined on $M$. Let $0<\gamma<1$. Assume that $\Delta_g+h$ is coercive, in the sense that its first eigenvalue is positive. Assume that $f>0$, $\rho_1>0$ and $|\nabla \rho_3|<(2C_1)^{-1}$, where $C_1$ is a dimensional constant - see (\ref{lxest}). Consider the system
\begin{equation}\label{systemshort}
\begin{cases}
\displaystyle \Delta_g u+hu
    &\displaystyle= fu^{q-1}+\frac{\rho_1+|\Psi+\rho_2\mathcal{L}_gW|^2_g}{u^{q+1}}\\ 
    &\displaystyle -\frac{b}{u}-c\langle\nabla u,Y\rangle \left(\frac{d}{u^2}+\frac{1}{u^{q+2}}\right)-\frac{\langle\nabla u,Y\rangle^2}{u^{q+3}}\\
    \displaystyle div_g\left(\rho_3\mathcal{L}_g W\right) 
    &\displaystyle =\mathcal{R}(u).
\end{cases}
\end{equation}
Here $\mathcal{R}$ is an operator verifying
$$\mathcal{R}(u)\leq C_\mathcal{R}\left(1+\frac{||u||_{\mathcal{C}^2}^2}{(\inf_Mu)^2}\right)$$
for a constant $C_\mathcal{R}>0$.
\par A supersolution of the Lichnerowicz-type equation is a smooth function $u$ verifying that
$$
\begin{array}{r l}
\Delta_g u+hu\geq &fu^{q-1}+\frac{\rho_1+|\Psi+\rho_2\mathcal{L}_gW|^2_g}{u^{q+1}} -\frac{b}{u}\\&
-c\langle\nabla u,Y\rangle \left(\frac{d}{u^2}+\frac{1}{u^{q+2}}\right)-\frac{\langle\nabla u,Y\rangle^2}{u^{q+3}}.
\end{array}
$$
Similarly, a subsolution satisfies an inequality of opposite sign. Whenever the inequality is strict, we say $u$ is a strict subsolution or a strict supersolution respectively.
\par We fix 
\begin{equation}\label{th main}
\theta=\min(\inf_M\rho_1,\inf_M f),
\end{equation}
and 
\begin{equation}\label{T main}
T=\max(||f||_{\mathcal{C}^{1,\gamma}}, ||\rho_1||_{\mathcal{C}^{0,\gamma}}, ||c||_{\mathcal{C}^{0,\gamma}}, ||d||_{\mathcal{C}^{0,\gamma}}, ||h||_{\mathcal{C}^{0,\gamma}}).
\end{equation} 
\par Here is the main result of our paper:
\begin{theo}\label{thm 1}
There exists a constant $C=C(n,h)$, $C>0$ such that if $\rho_1$ verifies
\begin{equation}
||\rho_1||_{L^1(M)}\leq C(n,h)\left(\max_M |f|\right)^{1-n},
\end{equation}
and there exists a constant $$\delta=\delta(\theta, T )>0$$ such that, if
\begin{equation}
||b||_{\mathcal{C}^{0,\gamma}}+ ||Y||_{\mathcal{C}^{0,\gamma}}+ ||\Psi||_{\mathcal{C}^{0,\gamma}}+||\rho_2||_{\mathcal{C}^{0,\gamma}}+ C_\mathcal{R}\leq \delta,
\end{equation}
then the system (\ref{systemshort}) admits a solution.
\end{theo}
\begin{rem} For a slightly more detailed expression of the smallness assumptions, see Section \ref{theo: system}. The constant $\displaystyle C(n,h)=\frac{C(n)}{S_h^{n-1}}$ appears explicitly in a paper by Hebey, Pacard and Pollack (\cite{HebPacPol08}, Corollary 3.1). By $S_h$ we understand the Sobolev constant which is defined as the smallest constant $S_h>0$ such that
\begin{equation*}
\int_M |v|^q\,dv_g\leq S_h\left(\int_M(|\nabla v|^2+hv^2)\,dv_g\right)^\frac{q}{2}
\end{equation*}
for all $v\in H^1(M)$.
\end{rem}
The following corollary deals with the existence of solutions to the conformal system. It suffices to take 
\begin{equation*}
\begin{cases}
h=\frac{n-2}{4(n-1)}\left(\mathcal{R}_g-|\nabla\psi|^2_g\right),\quad f=\frac{n-2}{4(n-1)}\left[2V(\psi)-\frac{n-1}{n}(\tau^*)^2\right],\\
\rho_1=\frac{n-2}{4(n-1)}\left(\pi-\frac{n-1}{n}(\tilde{N})^2div_g(\tilde{V})\right),\quad \rho_2=\sqrt{\frac{n-2}{4(n-1)}}\frac{\tilde{N}}{2}, \quad \Psi=\sqrt{\frac{n-2}{4(n-1)}}U,\\
b=\frac{n-2}{2n}\tau^*\tilde{N}div_g(\tilde{V}),\quad c=2\sqrt{\frac{n-2}{4n}},\quad d=\tau^*\\
Y=\sqrt{\frac{n}{n-2}}\tilde{N}\tilde{V},\quad \rho_3=\ln\tilde{N},\\
\mathcal{R}=\frac{n-1}{n}div_g(\tilde{V})\nabla\ln\tilde{N}+\frac{n-1}{n}\nabla(div_g(\tilde{V}))+\frac{\pi\delta_i\psi}{\tilde{N}}\\\quad+2\langle\tilde{V},\frac{\nabla u}{u}\rangle\nabla\ln\tilde{N}-2\frac{n-1}{n+1}\frac{\langle\tilde{V},\nabla u\rangle\nabla u}{u^2}-\frac{n-1}{n}\langle\tilde{V},\frac{\Delta_g u}{u}\rangle
\end{cases}
\end{equation*}
in (\ref{systemshort}). It is a direct application of Theorem \ref{thm 1}.
\begin{cor}\label{physical result} Let $\Delta_g + \frac{n-2}{4(n-1)}\left(\mathcal{R}_g-|\nabla\psi|^2_g\right)$ be a coercive operator. Assume that 
\begin{equation}
2V(\psi)>\frac{n-1}{n}(\tau^*)^2,\quad \pi>\frac{n-1}{n}(\tilde{N})^2div_g(\tilde{V})\quad\text{and}\quad|\nabla \ln\tilde{N}|<C_1^{-1},
\end{equation}
where $C_1$ depends on $n$ and $g$ (see (\ref{lxest}) for more details). Moreover, assume that 
\begin{equation}
||\pi-\frac{n-1}{n}(\tilde{N})^2div_g(\tilde{V})||_{L^1}\leq C(n,g,h)||2V(\psi)-\frac{n-1}{n}(\tau^*)^2||^{1-n}_{L^\infty}.
\end{equation}
Then there exists a constant
\begin{equation}
\begin{array}{c}
\delta=\delta\Big(\inf_M\frac{n-2}{4(n-1)}\left[2V(\psi)-\frac{n-1}{n}(\tau^*)^2\right],
\inf_M\frac{n-2}{4(n-1)}\left(\pi-\frac{n-1}{n}(\tilde{N})^2div_g(\tilde{V})\right)\\
\tau^*,||\pi||_{\mathcal{C}^{0,\gamma}},
||\mathcal{R}_g-|\nabla\psi|^2_g||_{\mathcal{C}^{0;\gamma}},||2V(\psi)||_{\mathcal{C}^{1,\gamma}}\Big)>0
\end{array}
\end{equation}
such that, if
\begin{equation}
||U||_{\mathcal{C}^{0,\gamma}}+||\pi||_{\mathcal{C}^{0,\gamma}}+||\nabla\psi||_{\mathcal{C}^{0,\gamma}}+||\ln \tilde N||_{\mathcal{C}^{1,\gamma}}+||\tilde{V}||_{\mathcal{C}^{1,\gamma}}\leq \delta,
\end{equation}
then (\ref{syst of Maxwell}) admits a solution $(u,W)$, where $u$ is a smooth positive function on $M$ and $W$ a smooth vector field on $M$.
\end{cor}
\bigskip
A few remarks on the results of the present paper. The classical system of constraint equations obtained by the conformal method (without the modifications proposed by Maxwell \cite{Max14}) was studied by Bruno Premoselli (\cite{Pre14}, \cite{Pre14a}) in the presence of a scalar field. Second, the above system is the subject of a paper by Mike Holst, David Maxwell and Rafe Mazzeo \cite{HolMaxMaz} - in their case, certain conditions are imposed on the presence of the matter field. We treat the separate and delicate case wherein the dominant non linearity is focusing and leads to possible loss of compactness. It is interesting to note that the size of $n$ plays a role; as Premoselli proves in his paper, while his system may be well-behaved in low dimensions ($3\leq n\leq 5$), it most certainly fails to do so in higher dimensions $(n\geq 6)$. Even if our results are similar to those of Premoselli, they are considerably more difficult to obtain. This is mainly due to the presence of a $|\nabla u|^2$ term in the scalar equation, a term which is not compact a priori.
\medskip
\par \textbf{Outline of the paper}. 
Section 2 is devoted to the study of the first equation in (\ref{systemshort}), the so-called Lichnerowicz equation. We prove the existence of stable solutions under suitable assumptions.
\par Section 3 deals with a priori estimates for solutions of the Lichnerowicz equation. A careful blow-up analysis is carried out. As already mentioned, the term $|\nabla u|^2$ poses additional difficulty: blow-up can occur at the $\mathcal{C}^1$ level, even if the solution is bounded in $L^\infty$. 
\par Section 4 is devoted to the proof of Theorem \ref{thm 1} and Corollary \ref{physical result}, which relies heavily on the a priori estimates obtained in Section 3. At the end of Section 4, we also explain how to extend Corollary \ref{physical result} in the presence of conformal Killing vector fields. 
\medskip
\par \textbf{Aknowledgements.} It is a pleasure to express my sincere gratitude to Olivier Druet for many helpful discussions and suggestions.
\section{Existence of minimal solutions of the scalar equation}
We study the Lichnerowicz-type scalar equation in (\ref{systemshort}). The following theorem states that, given the existence of supersolutions, one may use an iterative procedure to obtain a sequence which converges in $\mathcal{C}^1$ norm to a solution. We draw the reader's attention to the fact that this solution is uniquely determined by its construction. The proof contains some similarities with that of Premoselli \cite{Pre14}, but some new difficulties appear. The main difference here comes from the presence of non-linearities containing gradient terms, which force us to further refine the analysis. These gradient terms lead to difficulties in obtaining a priori estimates on solutions of the equation, which in turn lead to problems of stability. The existence result we prove in this section reads as follows:
\begin{theo}\label{theo: existence of sol to Lich type eq}
Let $(M,g)$ be a closed Riemannian manifold. Let $a$, $b$, $c$, $d$, $f$, $h$ be smooth functions on $M$ and $Y$ be a smooth vector field on $M$. Assume that $a>0$ and $f>0$. The equation
\begin{equation}\label{EM}
    \Delta_g u+ hu-fu^{q-1}-\frac{a}{u^{q+1}}+\frac{b}{u}+\frac{\langle\nabla u, Y\rangle^2}{u^{q+3}}+c\langle\nabla u,Y\rangle\left(\frac{d}{u^2}+\frac{1}{u^{q+2}}\right)=0
\end{equation}
admits a smooth positive solution $u$ as soon as it admits a supersolution.
\end{theo}
\begin{rem} The solution obtained by the construction below is unique. Moreover, it is stable (see Lemma \ref{lem weak stability} at the end of this section.)
\end{rem}
\begin{proof}[Proof of Theorem \ref{theo: existence of sol to Lich type eq}:]
We begin by fixing a supersolution and a subsolution to serve as upper and lower bounds respectively for the iterative process. Let $\psi$ be a positive supersolution of (\ref{EM}). Let $\varepsilon_0>0$ be a small constant such that 
$$\varepsilon_0<\inf_M\psi,\quad \left(\sup_M h\right)\varepsilon_0^{q+2}<\frac{\inf_M a}{2}\quad \text{and}\quad \left(\sup_M b\right)\varepsilon_0^{q}<\frac{\inf_M a}{2}.$$
The last two bounds ensure that $u_0=\varepsilon_0$ is a strict subsolution of (\ref{EM}) since $f>0$. 
We let
$$F(t,x)=-f(x)t^{q-1}-\frac{a(x)}{t^{q+1}}+\frac{b(x)}{t}-Kt$$
for $x\in M$ and $t>0$. We choose $K>0$ large enough such that
\begin{equation}
\label{monotonicity of F}
\frac{\partial}{\partial t}\left[F(t,x)+\frac{A(x)^2}{t^{q+3}}+c(x)A(x)\left(\frac{d(x)}{t^2}+\frac{1}{t^{q+2}}\right)\right]\leq 0
\end{equation}
for all $x\in M$ and all $\varepsilon_0\leq t\leq \sup_M\psi$, whatever $A(x)$ is. It is sufficient to take 
\begin{equation}
\label{K estimate}
\begin{array}{c}
K\geq \sup_{x\in M, \varepsilon_0\leq t\leq \sup_M\psi}\Big[-(q-1)f(x)t^{q-2}+(q+1)\frac{a(x)}{t^{q+2}}-\frac{b(x)}{t^2}\\
+\frac{c(x)^2\left(\frac{2d(x)}{t^3}+\frac{d+2}{t^{q+3}}\right)^2t^{q+4}}{4(q+3)}\Big].
\end{array}
\end{equation}
Up to choosing $K$ larger, we may also assume that $h+K>0$ and that $F(t,x)$ is negative for $\varepsilon_0\leq t\leq \sup_M\psi$. 
\par We shall now consider a sequence $(u_i)_{i\in\mathbb{N}}$ defined by induction by $u_0\equiv \varepsilon_0$ and
\begin{equation}\label{eq: the Ei eqs}
\begin{array}{c}
\displaystyle (E_i):\quad \Delta_g u_i+(h+K)u_i+F(u_{i-1}(x),x)+\frac{\langle\nabla u_i, Y\rangle^2}{u_{i-1}^{q+3}}\\ 
\displaystyle \quad +c\langle \nabla u_{i},Y\rangle\left(\frac{d}{u^2_{i-1}}+\frac{1}{u^{q+2}_{i-1}}\right)=0.
\end{array}
\end{equation}
\par We prove in Step 1 below that the sequence is well defined. In Step 2, we prove that the sequence if pointwise increasing and uniformly bounded. At last, Step 3 is devoted to the proof that the sequence $(u_i)$ converges to a solution of (\ref{eq: the Ei eqs}).
\par\textbf{Step 1:} We prove that $(u_i)$ is well defined. We consider the more general equation
\begin{equation}\label{gen eq}
\Delta_g u+Hu+\theta_1\langle\nabla u,Z\rangle^2+\theta_2\langle\nabla u,Z\rangle+\theta_3=0,
\end{equation}
with $H$, $\theta_1$, $\theta_2$, $\theta_3$ smooth functions on $M$ and $Z$ a smooth vector field on $M$ such that $\theta_1>0$, $H>0$, $\theta_3<0$. We claim that (\ref{gen eq}) admits a unique smooth positive solution.
\begin{proof}[Proof of Step 1:]
We shall use the fixed point theorem as stated in Evans \cite{Eva}, Section 9.2.2, Theorem 4. Let us define the operator $T:\mathcal{C}^{1,\gamma}(M)\to\mathcal{C}^{1,\gamma}(M)$ such that
$$\Delta_g T(u)+HT(u)+\theta_1\langle\nabla u,Z\rangle^2+\theta_2\langle\nabla u,Z\rangle+\theta_3=0.$$
 If we can prove that there exists $C>0$ such that
\begin{equation}\label{fixed point condition}
\forall\quad 0\leq \tau\leq 1,\quad w=\tau T(w)\quad\Rightarrow\quad ||w||_{\mathcal{C}^{1,\gamma}(M)}\leq C,
\end{equation}
then the operator $T$ will have a fixed point, leading to a solution of (\ref{gen eq}). Note that this solution will be unique. Indeed, assume that $w_1$ and $w_2$ are two solutions of (\ref{gen eq}), then at a point of maximum $x_0$ of $w_1-w_2$, we have that $\nabla w_1(x_0)=\nabla w_2(x_0)$ and $\Delta_g w_1(x_0)\geq \Delta_g w_2(x_0)$ so that (\ref{gen eq}) gives
$$H(x_0)\left(w_1(x_0)-w_2(x_0)\right)\leq 0.$$
Since $H>0$, this leads to $w_1\leq w_2$. By symmetry, uniqueness is proved. Note that the fixed point of $T$ is smooth and positive by the standard regularity theory and the maximum principle. 
\par Thus we are left with the proof of (\ref{fixed point condition}). Let $0\leq\sigma_m\leq 1$ and let $w_m\in\mathcal{C}^{1,\gamma}(M)$ be such that
$$w_m=\sigma_m T(w_m).$$
Multiplying (\ref{gen eq}) by $\sigma_m$, we obtain that
$$\Delta_g w_m+Hw_m+\sigma_m\theta_1\langle\nabla w_m,Z\rangle^2+\sigma_m\theta_2\langle\nabla w_m,Z\rangle+\sigma_m\theta_3=0.$$
First, the $L^\infty$ bounds on $w_m$ exist \textit{a priori}. Indeed, consider $x_0\in M$ a minimum of $w_m$. Since $\Delta_g w_m(x_0)\leq 0$ and $\nabla w_m(x_0)=0$, which holds true for all minima, then $Hw_m(x_0)\geq -\sigma_m \theta_3(x_0).$ By applying the same procedure to the study of maxima, we obtain that
\begin{equation}\label{a priori b}
\inf_M\frac{-\sigma_m\theta_3}{H}\leq w_m\leq \sup_M\frac{-\sigma_m\theta_3}{H}.
\end{equation}
Assume now that $||\nabla w_m||_{L^\infty(M)}\to\infty$. Let 
$$\mu_m:=\frac{1}{||\nabla w_m||_{L^\infty(M)}}\to 0\quad\text{ as }\quad m\to\infty,$$
and $(x_m)_m\subset M$ be such that
$$||\nabla w_m||_{L^\infty(M)}=|\nabla w_m(x_m)|.$$
Consider the domains $\displaystyle \Omega_m:=B_{
0}\left(\frac{i_g (M)}{2\mu_m}\right)$, where $i_g(M)$ is the injectivity radius of $M$, and the rescaled quantities 
$$v_m(x):=w_m\left(\exp_{x_m}(\mu_m x)\right)\quad\text{and}\quad g_m(x):=\left(\exp_{x_m}^*g\right)(\mu_m x).$$
Clearly, $||\nabla v_m||_{L^\infty}\leq 1$ and $|\nabla v_m(0)|=1$. The $L^\infty$ bounds remain unchanged. In $(\Omega_m)_{m\geq 1}$, we have that
\begin{equation*}
\begin{array}{c}
 \Delta_{g_m}v_m+\mu_m^2H\left(\exp_{x_m}(\mu_m\cdot)\right)v_m+\sigma_m\mu_m^2\theta_3\left(\exp_{x_m}(\mu_m\cdot)\right)\\ \\ 
+\sigma_m\theta_1\left(\exp_{x_m}(\mu_m\cdot)\right)\langle\nabla v_m,Z\left(\exp_{x_m}(\mu_m\cdot)\right)\rangle^2\\ \\
 +\mu_m\sigma_m\theta_2\left(\exp_{x_m}(\mu_m\cdot)\right)\langle\nabla v_m, Z\left(\exp_{x_m}(\mu_m\cdot)\right)\rangle=0
\end{array}
\end{equation*}
Note that $g_m\to\xi$ in $\mathcal{C}^2_{loc}(\R^n).$ By standard elliptic theory, $(v_m)_m$ is bounded in $\mathcal{C}^{1,\eta}_{loc}(\R^n)$, with $\eta\in (0,1)$. We may extract, up to a subsequence, $v_\infty=\lim_{m\to\infty}v_m$, $x_\infty=\lim_{m\to\infty}x_n$ and $\sigma_\infty:=\lim_{m\to\infty}\sigma_m$. From this is follows that $||\nabla v_\infty||_{L^\infty}=1$ and that the {\it  a priori } bounds (\ref{a priori b}) become
\begin{equation}\label{bounds on v infinity}
\inf_M\frac{-\sigma_{\infty}\theta_3}{H}\leq v_\infty\leq \sup_M\frac{-\sigma_{\infty}\theta_3}{H}.
\end{equation}
Moreover, $v_\infty$ solves the limit equation
$$\Delta v_\infty+\left(\partial_1 v_\infty\right)^2=0$$
in $\R^n$, where we have let
$$\partial_1 v_\infty:=\sqrt{\sigma_\infty\theta_1(0)}\nabla v_\infty\cdot Z(x_0).$$ 
If $\sigma_\infty=0$, then $v_\infty$ is a bounded harmonic function, and thus a constant. Let us assume that $\sigma_\infty\not = 0$. Note that, for $\alpha\in\R$,
\begin{equation*}
\begin{array}{r l}
\displaystyle \Delta v_\infty^{-\alpha} &\displaystyle =-\frac{\alpha}{v_\infty^{\alpha+1}}\Delta v_\infty-\frac{\alpha(\alpha+1)|\nabla v_\infty|^2}{v_\infty^{\alpha+2}}\\ 
&\displaystyle \leq\frac{\alpha|\nabla v_\infty|^2}{v_\infty^{\alpha+1}}\left(\sigma_\infty\theta_1(0)|Z(0)|^2-\frac{\alpha+1}{v_\infty}\right).
\end{array}
\end{equation*}
This and (\ref{bounds on v infinity}) imply that, for $\alpha$ sufficiently large, $v_\infty^{-\alpha}$ is subharmonic.  We then apply Lemma \ref{lem: A1} (see annex) to get that $v_\infty$ must be constant. Whichever the case, $\nabla v_\infty\equiv 0$ leads to a contradiction. The $\mathcal{C}^{1,\gamma}$ bound follows from an elliptic regularity argument. This ends the proof of Step 1.\end{proof} 
\par\textbf{Step 2: } We claim that
$$\varepsilon_0\leq u_i(x)\leq u_{i+1}(x)\leq \psi(x)$$
for all $x\in M$ and all $i\leq 0$.
\begin{proof}[Proof of Step 2:] We proceed by induction. We prove first that 
\begin{equation}\label{inductive hyp}
\forall\, i\geq 0,\quad u_i \text{ is a subsolution of } (E_{i+1}) \text{ and } u_i\leq \psi.
\end{equation}
Note that 
\begin{equation}\label{induction}
(\ref{inductive hyp}) \Rightarrow u_i\leq u_{i+1}.
\end{equation}
Indeed, let $x_0\in M$ be a maximum point of $u_i-u_{i+1}.$ Then $\nabla u_i(x_0)=\nabla u_{i+1}(x_0)$ and we can use the fact that $u_i$ is a subsolution of $(E_{i+1})$ and $u_{i+1}$ a solution of $(E_{i+1})$ to write that
$$\Delta_g(u_i-u_{i+1})(x_0)+(h(x_0)+K)(u_i-u_{i+1})(x_0)\leq 0$$
which implies that $u_i(x_0)\leq u_{i+1}(x_0)$ since $\Delta_g(u_i-u_{i+1})(x_0)\geq 0$ and $h+K>0$. This proves (\ref{induction}).
\par We now prove (\ref{inductive hyp}) by induction. For $i=0$, it follows from the choice of $\varepsilon_0$ we made. Assume that (\ref{inductive hyp}) holds for some $i\geq 0$. We need to prove that $u_{i+1}$ is a subsolution of $(E_{i+2})$. It suffices to show that
$$\begin{array}{c}
\Delta_g u_{i+2}+(h+K)u_{i+2}+F(u_{i+1}(x),x)+\frac{\langle\nabla u_{i+2}, Y\rangle^2}{u_{i+1}^{q+3}}\\ 
 +c\langle \nabla u_{i+2},Y\rangle\left(\frac{d}{u^2_{i+1}}+\frac{1}{u^{q+2}_{i+1}}\right)\leq \Delta_g u_{i+1}+(h+K)u_{i+1}+F(u_{i+1}(x),x)\\+\frac{\langle\nabla u_{i+1}, Y\rangle^2}{u_{i+1}^{q+3}} +c\langle \nabla u_{i+1},Y\rangle\left(\frac{d}{u^2_{i+1}}+\frac{1}{u^{q+2}_{i+1}}\right),
\end{array}
$$
since $u_{i+1}$ is defined as a solution of $(E_{i+1})$. This is equivalent to showing that
$$\begin{array}{c}
F(u_{i+1})+c\langle \nabla u_{i+1},Y\rangle\left(\frac{d}{u_{i+1}^2}+\frac{1}{u_{i+1}^{q+2}}\right)+\frac{\langle\nabla u_{i+1},Y\rangle^2}{u_{i+1}^{q+3}}\\ \leq F(u_{i})+c\langle\nabla u_{i+1},Y\rangle\left(\frac{d}{u_{i}^2}+\frac{1}{u_{i}^{q+2}}\right)+\frac{\langle\nabla u_{i+1},Y\rangle^2}{u_{i}^{q+3}}.
\end{array}
$$
And this is a consequence of (\ref{monotonicity of F}) with $A(x)=\langle\nabla u_{i+1},Y\rangle$, since (\ref{induction}) implies that $u_{i+1}\geq u_i$ by induction hypothesis. Thus, $u_{i+1}$ is a subsolution of $(E_{i+2})$.
\par Finally, so as to check the last point, assume there exists $x_0\in M$ such that $u_{i+1}(x_0)>\psi(x_0)$ and that it corresponds to $\max_M\left(u_{i+1}(x)-\psi(x)\right).$ Since $\nabla u_{i+1}(x_0)=\nabla \psi (x_0)$ and $\Delta_g u_{i+1}(x_0)\geq \Delta_g\psi(x_0),$ we obtain that
$$\Delta_g(u_{i+1}-\psi)(x_0)+(h+K)(u_{i+1}-\psi)(x_0)>0.$$
But $\psi$ is a supersolution for (\ref{EM}), so we get that
\begin{equation*}
\begin{array}{c}
 0<\Delta_g(u_{i+1}-\psi)(x_0)+(h(x_0)+K)(u_{i+1}-\psi)(x_0)\\ 
 \leq F(\psi(x_0),x_0)-F(u_i(x_0),x_0)-\langle\nabla\psi(x_0),Y(x_0)\rangle^2\left(\frac{1}{u_i^{q+3}}-\frac{1}{\psi^{q+3}}\right)(x_0)\\
 -\langle\nabla\psi(x_0),Y(x_0)\rangle\left(\frac{d}{u_i^2}-\frac{d}{\psi^2}+\frac{1}{u_i^{q+2}}-\frac{1}{\psi^{q+2}}\right)(x_0).
\end{array}
\end{equation*}
Thanks to (\ref{monotonicity of F}) with $A(x)=\langle\nabla\psi(x),Y(x)\rangle$ and to the induction hypothesis which says that $u_i\leq\psi$, we obtain a contradiction. This wraps up the induction argument and the proof of Step 2.\end{proof}
\par\textbf{Step 3:} The sequence $(u_i)_{i\in\mathbb{N}}$ is uniformly bounded in $\mathcal{C}^1(M)$.
\begin{proof}[Proof of Step 3:] Thanks to Step 2, we know that $(u_i)_{i\in\mathbb{N}}$ is an increasing sequence bounded by $\psi$. Thus there exists $u\in\mathcal{C}^0(M)$ such that $u_i\to u_0$ in $\mathcal{C}^0(M).$
\par Assume by contradiction that exists a subsequence $(u_{\phi(m)})_{m\in N}$ such that $||\nabla u_{\phi(m)}||_{L^\infty}\to\infty.$ Let 
$$\displaystyle \mu_m:=\frac{1}{||\nabla u_{\phi(m)}||_{L^\infty}}$$
and let $(x_m)_m\subset M$ be such that
$$|\nabla u_{\phi(m)}(x_m)|=||\nabla u_{\phi(m)}||_{L^\infty}.$$
Consider the domains $\displaystyle \Omega_m=B_{0}\left(\frac{i_g M}{2\mu_m}\right)$ and the rescaled quantities
$$v_m(x):=u_{\phi(m)}\left(\exp_{x_m}(\mu_m x)\right)\quad\text{ and }\quad g_m(x):=\left(\exp_{x_m}^* g\right)(\mu_m x)$$
in $\Omega_m$. We get
\begin{equation}\label{eq: scaling in exist thm}
\begin{array}{c}
\Delta_{g_m}v_m+\mu_m^2(h\left(\exp_{x_m}(\mu_m \cdot)\right)+K)v_m\left(\exp_{x_m}(\mu_m \cdot)\right)\\ 
\displaystyle +\mu_m^2F\left(u_{\phi(m)-1}\left(\exp_{x_m}(\mu_m \cdot)\right)\right)
\displaystyle +\frac{\langle\nabla v_m,Y\left(\exp_{x_m}(\mu_m \cdot)\right)\rangle^2}{u_{\phi(m)-1}^{q+3}\left(\exp_{x_m}(\mu_m \cdot)\right)}\\ 
\displaystyle +\mu_m\langle\nabla v_m, Y\left(\exp_{x_m}(\mu_m \cdot)\right)\rangle c\left(\exp_{x_m}(\mu_m \cdot)\right)\Big[\frac{d\left(\exp_{x_m}(\mu_m \cdot)\right)}{u_{\phi(m)-1}^2\left(\exp_{x_m}(\mu_m \cdot)\right)}\\ 
\displaystyle +\frac{1}{u_{\phi(m)-1}^{q+2}\left(\exp_{x_m}(\mu_m \cdot)\right)}\Big]=0 
\end{array}
\end{equation}
with $(v_m)_{m\in N}$ bounded in $L^\infty$, $||\nabla v_m||_{L^\infty}=1$ and $\varepsilon_0\leq v_m$. By the Sobolev embedding theorem and standard elliptic regularity, there exists a smooth positive limit $v_\infty$ of $(v_m)_{m\in N}$, up to a subsequence. Recall that $(u_i)_{i\in N}$ converges everywhere, so $u_{\phi(m)-1}(\mu_m x)\to v_\infty(x)$ in $M$. By taking $m\to\infty$ in (\ref{eq: scaling in exist thm}),
$$\Delta v_\infty+\frac{\left(\nabla v_\infty\cdot Y(0)\right)^2}{v_\infty^{q+3}}=0.$$
Note also that 
$$\Delta v_\infty^{-\alpha}\leq \frac{\alpha|\nabla v_\infty|^2}{v_\infty^{\alpha+2}}\left(\frac{|Y(0)|^2}{v_\infty^{q+2}}-(\alpha+1)\right).$$
For $\alpha$ large enough, $v_\infty^{-\alpha}$ is subharmonic. Using Lemma \ref{lem: A1} (see Annex), we find that $v_\infty$ is constant, which contradicts the fact that $||\nabla v_\infty||_{L^\infty}=1.$ This ends the proof of Step 3.\end{proof}
\par Since $(u_i)_{i\in N}$ is uniformly bounded in $\mathcal{C}^1$, we conclude by standard elliptic theory that its limit $u$ is a positive smooth function solving equation (\ref{EM}). This ends the proof of the theorem.
\end{proof}
The solution constructed in the previous proof is uniquely determined as the pointwise limit of $(u_i)_{i\in N}$, where each $u_i$ is the unique solution of (\ref{eq: the Ei eqs}). Furthermore, the solution is minimal among all supersolutions (including solutions) of (\ref{EM}) with values between $\varepsilon_0$ and $\sup_M\psi$. These bounds were explicitly used in the inductive argument. By construction, $u\leq\psi$, where $\psi$ is the supersolution fixed at the very beginning. Note that the constant $K$ appearing in (\ref{eq: the Ei eqs}) depend on $\sup_M\psi$ and $\varepsilon_0$. We would obtain the same iteration were we to use another supersolution $\tilde{\psi}$ and the same $K$, given that $\varepsilon_0<\tilde{\psi}<\sup_M\psi$. Therefore, $u$ is smaller than any supersolution between $\varepsilon_0$ and $\sup_M\psi$.
\par
As an immediate consequence of the minimality discussed above, the solutions we found corresponding to different functions $a$ are ordered. Let $0<a<\tilde{a}$ be two functions, and assume that the equation associated to $\tilde{a}$ admits a solution $\tilde{u}$. Then $\tilde{u}$ is a supersolution for (\ref{EM}) corresponding to $a$, and by the previous proof we find a solution $u\leq\tilde{u}$. Moreover, given that $\tilde{u}$ may be viewed as a supersolution to all (\ref{EM}) with $a\leq\tilde{a}$, we obtain a monotonicity of $u$ in $a$: for $a_1\leq a_2\leq \tilde{a}$, then $u_1\leq u_2\leq \tilde{u}$.
\par Finally, the solution $u$ is stable, as defined in the following lemma.
\begin{lemma}\label{lem weak stability}
The operator $L$ resulting from the linearization of (\ref{EM}) at the minimal solution $u$ admits a real, simple eigenvalue $\lambda_0\geq 0$ such that
\begin{equation*}
    \begin{array}{c}
    \displaystyle L\varphi_0=\Delta_g\varphi_0+\Big[h-(q-1)fu^{q-2}+(q+1)au^{-q-2}-bu^{-2}\\
    -(q+3)\frac{\langle\nabla u,Y\rangle^2}{u^{q+4}}-c\langle\nabla u,Y\rangle\left(\frac{2d}{u^3}+\frac{q+2}{u^{q+3}}\right)\Big]\varphi_0\\+\langle\nabla\varphi_0,Y\rangle\Big[c\left(\frac{d}{u^2}+\frac{1}{u^{q+2}}\right)+\frac{2\langle\nabla u,Y\rangle}{u^{q+3}}\Big]\\=\lambda_0\varphi_0,
    \end{array}
\end{equation*}
where $\varphi_0$  is the corresponding positive eigenfunction. Furthermore, if $\lambda\in \C$ is any other eigenvalue, then $Re(\lambda)\geq \lambda_0$. 
\end{lemma}
\begin{proof}[Proof of Lemma \ref{lem weak stability}:] Notice that $L$ is nonsymmetric; moreover, one may find a large enough constant $K$ such that
$$
\begin{array}{c}
h-(q-1)fu^{q-2}+(q+1)au^{-q-2}-bu^{-2}-(q+3)\frac{\langle\nabla u,Y\rangle^2}{u^{q+4}}\\
-c\langle\nabla u,Y\rangle\left(\frac{2d}{u^3}+\frac{q+2}{u^{q+3}}\right)+K\geq 0.
\end{array}
$$
According to (\cite{Eva}, Section 6.5, Theorem 1) there exists a real, positive eigenvalue $\lambda_K>0$ of $L+K$, such that any other complex eigenvalue of $L+K$ has a greater real part. Consequently, the operator $L$ admits a minimal real eigenvalue $\lambda_0>-K$. We now assume that $\lambda_0<0$. Let $u_\delta:=u_0-\delta\varphi_0$, $\delta>0$. By taking $\delta$ small enough, we may ensure that $\varepsilon_0< u_\delta$. Then
$$\begin{array}{c}
\displaystyle \Delta_g u_\delta+hu_\delta-fu_\delta^{q-1}-\frac{a}{u_\delta^{q+1}}+\frac{b}{u_\delta}+\frac{\langle\nabla u_\delta,Y\rangle^2}{u_\delta^{q+3}}+c\langle\nabla u_\delta, Y\rangle\left(\frac{d}{u_\delta^2}+\frac{1}{u_\delta^{q+2}}\right)\\ \\
\displaystyle =-\delta\lambda_0\varphi_0+o(\delta).
\end{array}
$$
This implies that $\varepsilon_0< u_\delta<u$ is a supersolution of (\ref{EM}), which cannot be the case, as discussed above. Thus, $\lambda_0\geq 0.$
\end{proof}

\section{A priori estimates on solutions of the scalar equation in low dimensions}
The $\mathcal{C}^1$ estimates obtained in this section will play a crucial role in the proof of Theorem \ref{thm 1}, which is based on a fixed-point argument. This section is devoted to the proof of the following theorem: 
\begin{theo}\label{theo stability eq}
Let $(M,g)$ be a closed Riemannian manifold of dimension $n=3,4,5$. Let $\frac{1}{2}<\eta<1$ and $0<\alpha<1$. Let $a$, $b$, $c$, $d$, $f$, $h$ be smooth functions on $M$, let $Y$ be a smooth vector field on the $M$. 
\par For any $0<\theta<T$, there exists $S_{\theta,T}$ such that any smooth positive solution $u$ of (\ref{EM}) with parameters within
$$
\begin{array}{c}
\mathcal{E}_{\theta,T}:=\Big\{(f,a,b,c,d,h,Y),\quad f\geq \theta,\quad a\geq \theta, \\
\quad ||f||_{\mathcal{C}^{1,\eta}}\leq T,\quad ||a||_{\mathcal{C}^{0,\alpha}},||b||_{\mathcal{C}^{0,\alpha}}, ||c||_{\mathcal{C}^{0,\alpha}}, ||d||_{\mathcal{C}^{0,\alpha}}, ||h||_{\mathcal{C}^{0,\alpha}}, ||Y||_{\mathcal{C}^{0,\alpha}}\leq T\Big\},
\end{array}
$$
satisfies $||u||_{\mathcal{C}^2}\leq S_{\theta,T}$.
\end{theo}
\begin{rem} For the sake of clarity, we've taken the bounds on the parameters to be of the form $\theta$ and $T$. They can of course be individually specified.
\end{rem}We proceed by contradiction. We assume the existence of a sequence $(u_\alpha)_{\alpha\in\N}$ of smooth positive solutions of equations ($EL_\alpha$)
\begin{equation}\label{eq: EL alpha}
\begin{array}{c}
    \Delta_g u_\alpha+ h_\alpha u_\alpha-f_\alpha u_\alpha^{q-1}- \frac{a_\alpha}{u_\alpha^{q+1}}+\frac{b_\alpha}{u_\alpha}+\frac{\langle\nabla u_\alpha, Y_\alpha\rangle^2}{u_\alpha^{q+3}}\\
     +c_\alpha\langle\nabla u_\alpha,Y_\alpha\rangle\left[\frac{d_\alpha}{u_\alpha^2}+\frac{1}{u_\alpha^{q+2}}\right]=0
\end{array}
\end{equation}
with parameters $(a_\alpha, b_\alpha, c_\alpha, d_\alpha, f_\alpha,h_\alpha, Y_\alpha)$ in $\mathcal{E}_{\theta,T}$ such that
\begin{equation}
\label{est: explosion}
||u_\alpha||_{\mathcal{C}^1(M)}\to\infty,\quad\text{ as }\alpha\to\infty.
\end{equation}
A \textit{concentration point} is the limit in $M$ of any sequence $(x_\alpha)_\alpha$ where (\ref{est: explosion}) holds. Note that a $\mathcal{C}^1$-bound on $(u_\alpha)_\alpha$ automatically gives a $\mathcal{C}^2$-bound by elliptic theory. Note also that, up to a subsequence, all parameters converge in $\mathcal{C}^0(M).$
\par Let $m_\alpha=\min_{x\in M}u_\alpha(x)=u_\alpha(x_\alpha)>0.$ Since $\nabla u_\alpha(x_\alpha)=0$ and since $\Delta_g u_\alpha(x_\alpha)\leq 0$, we have thanks to (\ref{eq: EL alpha}) that
$$h_\alpha(x_\alpha)m_\alpha-f_\alpha(x_\alpha)m_\alpha^{q-1}-\frac{a_\alpha(x_\alpha)}{m_\alpha^{q+1}}+\frac{b_\alpha(x_\alpha)}{m_\alpha}\geq 0.$$
Thanks to the definition of $\mathcal{E}_{\theta,T},$ it follows that
$$\frac{\theta}{m_\alpha^{q+1}}\leq T(m_\alpha+\frac{1}{m_\alpha}).$$
Then there exists $\varepsilon=\varepsilon(\theta,T,n)>0$ such that $m_\alpha\geq \varepsilon$, meaning that
\begin{equation}\label{varepsilon}
u_\alpha>\varepsilon>0\quad \text{ for all } x\in M \text{ and all } \alpha.
\end{equation}
\par The scheme of the proof follows the work of Druet and Hebey \cite{DruHeb}, with the added difficulty consisting in the gradient terms in (\ref{eq: EL alpha}).
\subsection{Concentration points}
The first step in finding \textit{a priori} estimates for $(u_\alpha)_\alpha$ is to find all potential concentration points. 
\begin{lemma}\label{lem: former lemma 2}
There exists $N_\alpha\in \N^*$ and $\displaystyle \mathcal{S}_\alpha:=\left(x_{1,\alpha},\dots x_{N_\alpha, \alpha}\right)$ a set of critical points of $(u_\alpha)_\alpha$ such that
\begin{equation}\label{est: 1 of former lemma 1}
d_g(x_{i,\alpha},x_{j,\alpha})^{\frac{n-2}{2}}u_\alpha(x_{i,\alpha})\geq 1
\end{equation}
for all $i,j\in\{1,\dots,N_\alpha\}$, $i\not =j$, and
\begin{equation}\label{est: add fact}
\left(\min_{i=1,\dots, N_\alpha}d_g(x_{i,\alpha},x)\right)^{\frac{n-2}{2}}u_\alpha(x)\leq 1
\end{equation}
for all critical points of $u_\alpha$ and such that there exists $C_1>0$ such that
\begin{equation}\label{est: initial estimate}\left(\min_{i=1,\dots N_\alpha} d_g(x_{i,\alpha},x)\right)^{\frac{n-2}{2}}\left(u_\alpha(x)+\left|\frac{\nabla u_\alpha(x)}{u_\alpha(x)}\right|^{\frac{n-2}{2}}\right)\leq C_1\end{equation}
for all $x\in M$ and all $\alpha\in \N$.
\end{lemma}
\begin{rem}
Estimate (\ref{est: initial estimate}) implies that any concentration point of $(u_\alpha)_{\alpha\in N}$ calls for the existence of a sequence $(x_\alpha)_\alpha\subset (\mathcal{S}_\alpha)_\alpha$ converging to it. We shall focus our analysis in the neighbourhood of $S_\alpha$ as $\alpha\to\infty$ to find concentration points.
\end{rem}
\begin{proof}[Proof of Lemma \ref{lem: former lemma 2}:] 
In order to choose $(S_\alpha)_\alpha$, we make use of a simple result describing any sufficiently regular function on a compact manifold. 
\begin{lemma}\label{lem: former lemma 1}
Let $u$ be a positive real-valued $\mathcal{C}^2$ function defined in a compact manifold $M$. Then there exists $N\in \N^*$ and $\left(x_1,\, x_2,\dots x_N\right)$ a set of critical points of $u$ such that
$$
    d_g(x_i,x_j)^{\frac{n-2}{2}}u(x_i)\geq 1
$$
for all $i,\,j\in\{1, \dots, N\}$, $i\not= j$, and 
$$\left(\min_{i=1,\dots, N}d_g(x_i,x)\right)^{\frac{n-2}{2}}u(x)\leq 1$$
for all critical points $x$ of $u$.
\end{lemma}
The lemma and its proof may be found in Druet and Hebey's paper \cite{DruHeb}. Applying this lemma to $(u_\alpha)$ gives $N_\alpha$ and $\mathcal{S}_\alpha$ as in Lemma \ref{lem: former lemma 2} such that (\ref{est: 1 of former lemma 1}) and (\ref{est: add fact}) hold. We need to prove (\ref{est: initial estimate}).
Proceeding by contradiction, assume that there exists a sequence $(x_\alpha)_\alpha$ such that 
\begin{equation}\label{est: contradiction initial estimate}\left(\min_{i=1,\dots N_\alpha} d_g(x_{i,\alpha},x_\alpha)\right)^{\frac{n-2}{2}}\left(u_\alpha(x_\alpha)+\left|\frac{\nabla u_\alpha(x_\alpha)}{u_\alpha(x_\alpha)}\right|^{\frac{n-2}{2}}\right)\to\infty\end{equation}
as $\alpha\to\infty$, where
\begin{equation}\label{contr}
\begin{array}{c}
 \displaystyle \left(\min_{i=1,\dots N_\alpha} d_g(x_{i,\alpha},x_\alpha)\right)^{\frac{n-2}{2}}\left(u_\alpha(x_\alpha)+\left|\frac{\nabla u_\alpha(x_\alpha)}{u_\alpha(x_\alpha)}\right|^{\frac{n-2}{2}}\right)\\ \\
 \displaystyle =\sup_{x\in M}\left(\min_{i=1,\dots N_\alpha} d_g(x_{i,\alpha},x)\right)^{\frac{n-2}{2}}\left(u_\alpha(x)+\left|\frac{\nabla u_\alpha(x)}{u_\alpha(x)}\right|^{\frac{n-2}{2}}\right).
\end{array}
\end{equation}

Denote $$\nu_\alpha^{1-\frac{n}{2}}:=u_\alpha(x_\alpha)+\left|\frac{\nabla u_\alpha(x_\alpha)}{u_\alpha(x_\alpha)}\right|^{\frac{n-2}{2}}$$ 
and see that (\ref{est: contradiction initial estimate}) translates to
\begin{equation}\label{est: distance to s alpha}
    \frac{d_g(x_\alpha, \mathcal{S}_\alpha)}{\nu_\alpha}\to\infty\quad\text{as}\quad \alpha\to\infty.
\end{equation}
Also, since $M$ is compact, 
\begin{equation}\label{nu goes to zero}
\nu_\alpha\to 0 \text{ as }\alpha\to\infty.
\end{equation}
Consider the rescaled quantities
$$v_\alpha(x):=\nu_\alpha^{\frac{n-2}{2}}u_\alpha\left(\exp_{x_\alpha}(\nu_\alpha x)\right)\quad\text{and}\quad g_\alpha(x):=\left(\exp_{x_\alpha}^* g\right)(\nu_\alpha x)$$
 defined in $\Omega_\alpha:=B_0\left(\frac{\delta}{\nu_\alpha}\right)$, with $ 0<\delta<\frac{1}{2}i_g(M).$ We emphasize that, for any $R>0$, 
\begin{equation}\label{est: size of v alpha and co}
    \limsup_{\alpha\to\infty}\sup_{B_0(R)} \left(v_\alpha+\left|\frac{\nabla v_\alpha}{v_\alpha}\right|^{\frac{n-2}{2}}\right) =1
\end{equation}
thanks to (\ref{contr}) and (\ref{est: distance to s alpha}). However, unlike $(u_\alpha)_\alpha$, the sequence $(v_\alpha)_\alpha$ is not necessarily bounded from below by a small positive constant $\varepsilon$. Instead, we deduce from (\ref{est: size of v alpha and co}) that 
$$|\nabla \ln v_\alpha|\leq 1+o(1)\quad\text{ in } B_0(R)$$
for all $R>0$ so that
\begin{equation}\label{size of v alpha}
    v_\alpha(0)e^{-2|x|}\leq v_\alpha(x)\leq v_\alpha(0)e^{2|x|}\quad\text{ in  } B_0(R)
\end{equation}
for all $R>0$ as soon as $\alpha$ is large enough.
We rewrite (\ref{eq: EL alpha}) in $\Omega_\alpha$ as
\begin{equation}\label{eq: Ev alpha eqs}
    \begin{array}{c}
          \Delta_{g_\alpha}v_\alpha = f_\alpha\left(\exp_{x_\alpha}(\nu_\alpha\cdot)\right)v_\alpha^{q-1}+\nu_\alpha^{\frac{n+2}{2}}\frac{a_\alpha\left(\exp_{x_\alpha}(\nu_\alpha\cdot)\right)}{u_\alpha^{q+1}\left(\exp_{x_\alpha}(\nu_\alpha\cdot)\right)} \\ 
          -\nu_\alpha^2 h_\alpha\left(\exp_{x_\alpha}(\nu_\alpha\cdot)\right)v_\alpha-\nu_\alpha^{\frac{n+2}{2}}\frac{b_\alpha \left(\exp_{x_\alpha}(\nu_\alpha\cdot)\right)}{u_\alpha \left(\exp_{x_\alpha}(\nu_\alpha\cdot)\right)}\\ 
      -\frac{\langle\nabla v_\alpha,Y_\alpha\left(\exp_{x_\alpha}(\nu_\alpha\cdot)\right)\rangle^2}{v_\alpha}\frac{1}{u_\alpha^{q+2}\left(\exp_{x_\alpha}(\nu_\alpha\cdot)\right)}\\ 
      -\nu_\alpha c_\alpha\left(\exp_{x_\alpha}(\nu_\alpha\cdot)\right)\langle\nabla v_\alpha,Y_\alpha\left(\exp_{x_\alpha}(\nu_\alpha\cdot)\right)\rangle\Big[\frac{d_\alpha \left(\exp_{x_\alpha}(\nu_\alpha\cdot)\right)}{u_\alpha^2\left(\exp_{x_\alpha}(\nu_\alpha\cdot)\right)}\\ 
      +\frac{1}{u_\alpha^{q+2}\left(\exp_{x_\alpha}(\nu_\alpha\cdot)\right)}\Big]
    \end{array}
\end{equation}
Note that the metrics $g_\alpha\to\xi$ in $\mathcal{C}^2_{loc}$ as $\alpha\to\infty$. Because of (\ref{varepsilon}) and (\ref{est: size of v alpha and co}), the right hand side is bounded, so by standard elliptic theory there exists up to a subsequence a $\mathcal{C}^1$ limit $U:=\lim_{\alpha\to\infty}v_\alpha$ and $x_0:=\lim_{\alpha\to\infty}x_\alpha$.
\par \textbf{First case:} If $u_\alpha(x_\alpha)\to\infty$ as $\alpha\to\infty$, then we can pass to the limit in equation (\ref{eq: Ev alpha eqs}) to get
$$\Delta U=f(x_0)U^{q-1}$$
in $\R^n$. The exact form of these solutions is found in a paper by Caffarelli, Gidas and Spruck \cite{CaffGidSpr}:
$$U(x)=\left(1+\frac{f(x_0)|x-y_0|^2}{n(n-2)}\right)^{1-\frac{n}{2}}$$
with $y_0\in \R^n$ the unique maximum point of the function. There exist therefore $(y_\alpha)_\alpha$ local maxima of $(u_\alpha)_\alpha$ such that 
\begin{equation}\label{est: blow up analysis, distances}
    d_g(x_\alpha,y_\alpha)=O(\nu_\alpha)
\end{equation}
and
\begin{equation}\label{est lim 1}
    \nu_\alpha^{\frac{n-2}{2}}u_\alpha(y_\alpha)\to 1\quad \text{ as }\quad \alpha\to\infty.
\end{equation}
Since $(y_\alpha)_\alpha$ are critical points, (\ref{est: add fact}) implies that
$$d_g(\mathcal{S}_\alpha, y_\alpha)^{\frac{n-2}{2}}u_\alpha(y_\alpha)\leq 1$$
for all $\alpha\in \N$, so by (\ref{est lim 1}), $\displaystyle d_g(\mathcal{S}_\alpha,y_\alpha)=O(\nu_\alpha)$; together with (\ref{est: blow up analysis, distances}), this leads to $\displaystyle d_g(\mathcal{S}_\alpha,x_\alpha)=O(\nu_\alpha)$, which contradicts (\ref{est: distance to s alpha}).
\par\textbf{Second case:} Assume that, up to a subsequence, $u_\alpha(x_\alpha)\to l<\infty$. From (\ref{nu goes to zero}) we deduce that $|\nabla u_\alpha (x_\alpha)|\to\infty$ and that $v_\alpha(0)\to 0$ as $\alpha\to\infty$. Let us set
$$w_\alpha(x):=\frac{v_\alpha(x)}{v_\alpha(0)}.$$
These functions are bounded from below, since by (\ref{varepsilon}),
\begin{equation}\label{est w}
w_\alpha(x)=\frac{u_\alpha\left(\exp_{x_\alpha}(\nu_\alpha\cdot)\right)}{u_\alpha(x_\alpha)}\geq\frac{\varepsilon}{l}+o(1)
\end{equation}
in $B_0(R)$ for all $R>0$. Moreover, (\ref{size of v alpha}) implies that
$$w_\alpha(x)\leq e^{2|x|}$$
in $B_0(R)$ for $\alpha$ large. Multiply (\ref{eq: Ev alpha eqs}) by $v_\alpha(0)^{-1}$ to get
\begin{equation*}
    \begin{array}{r l}
    \displaystyle \Delta_g w_\alpha=     &\displaystyle f_\alpha\left(\exp_{x_\alpha}(\nu_\alpha\cdot)\right)w_\alpha^{q-1}v_\alpha^{q-2}(0)+\nu_\alpha^2\frac{a\left(\exp_{x_\alpha}(\nu_\alpha\cdot)\right)}{u_\alpha^{q+1}\left(\exp_{x_\alpha}(\nu_\alpha\cdot)\right)u_\alpha(x_\alpha)}   \\
         & \displaystyle -\nu_\alpha^2h_\alpha\left(\exp_{x_\alpha}(\nu_\alpha\cdot)\right)w_\alpha-\nu_\alpha^2\frac{b_\alpha\left(\exp_{x_\alpha}(\nu_\alpha\cdot)\right)}{u_\alpha\left(\exp_{x_\alpha}(\nu_\alpha\cdot)\right)u_\alpha(x_\alpha)}\\ 
         & \displaystyle-\frac{\langle\nabla w_\alpha,Y_\alpha\left(\exp_{x_\alpha}(\nu_\alpha\cdot)\right)\rangle^2}{w_\alpha^{q+2}}\frac{1}{u_\alpha(x_\alpha)^{q+2}}\\
         &\displaystyle -\nu_\alpha c_\alpha\left(\exp_{x_\alpha}(\nu_\alpha\cdot)\right)\langle\nabla w_\alpha,Y_\alpha\left(\exp_{x_\alpha}(\nu_\alpha\cdot)\right)\rangle\Big[\frac{d_\alpha \left(\exp_{x_\alpha}(\nu_\alpha\cdot)\right)}{u_\alpha^2\left(\exp_{x_\alpha}(\nu_\alpha\cdot)\right)}\\ 
         &\displaystyle +\frac{1}{u_\alpha^{q+2}\left(\exp_{x_\alpha}(\nu_\alpha\cdot)\right)}\Big].
    \end{array}
\end{equation*}
    By standard elliptic theory, we find that there exists $w:=\lim_{\alpha\to\infty}w_\alpha$ in $\mathcal{C}^1$ solving:
    $$\Delta w=-\frac{1}{l^{q+2}}\frac{\langle\nabla w, Y(x_0)\rangle^2}{w^{q+3}}$$
    in $\R^n$. Note that, since $w\geq \frac{\varepsilon}{l}$ by (\ref{est w}), 
    $$\Delta w^{-\alpha}\leq \alpha\frac{|\nabla w|^2}{w^{\alpha+2}}\left[\frac{|Y(x_0)|^2}{\varepsilon^{q+2}}-(\alpha+1)\right],$$
    so $w^{-\alpha}$ is subharmonic for $\alpha$ large. By applying Lemma \ref{lem: A1} (see the Annex), we deduce that $w$ is constant, which in turn implies that $U=0$, and so $\nabla U=0$. This implies that
$$\lim_{\alpha\to\infty}\nu_\alpha\frac{\nabla u_\alpha(x_\alpha)}{u_\alpha(x_\alpha)}=0,$$    
     as it contradicts the choice of $\nu_\alpha$ above, which may be rewritten as
     \begin{equation*}
     \nu_\alpha^\frac{n-2}{2}\left(u_\alpha(x_\alpha)+\left|\frac{\nabla u_\alpha(x_\alpha)}{u_\alpha(x_\alpha)}\right|\right)=1,
     \end{equation*}
since we are in the case where $\nu_\alpha^{\frac{n-2}{2}}u_\alpha(x_\alpha)=v_\alpha(0)\to 0$ as $\alpha\to 0.$
\end{proof}

\par The following is a Harnack-type inequality. It holds whenever an estimate like (\ref{est: initial estimate}) is verified, that is when there exists a constant $C_2$ and a sequence $(x_\alpha,\rho_\alpha)_\alpha$ such that 
\begin{equation}\label{est: for Harnack}
   d_g(x_\alpha,x)^{\frac{n-2}{2}}\left[u_\alpha(x)+\left|\frac{\nabla u_\alpha(x)}{u_\alpha(x)}\right|^{\frac{n-2}{2}}\right]\leq C_2,\quad\forall x\in B_{x_\alpha}(7\rho_\alpha).
\end{equation}
\begin{lemma}\label{lem: former lemma 3}
Let $(x_\alpha, \rho_\alpha)_\alpha$ be a sequence such that (\ref{est: for Harnack}) holds. Then there exists a constant $C_3>1$ such that for any sequence $0<s_\alpha\leq\rho_\alpha$, we get
$$s_\alpha||\nabla u_\alpha||_{L^\infty(\Omega_\alpha)}\leq C_3\sup_{\Omega_\alpha}u_\alpha\leq C_3^2\inf_{\Omega_\alpha}u_\alpha,$$
where $\Omega_\alpha=B_{x_\alpha}(6s_\alpha)\backslash B_{x_\alpha}(\frac{1}{6}s_\alpha).$
\end{lemma}
\begin{proof}[Proof of Lemma \ref{lem: former lemma 3}:]
Estimate (\ref{est: for Harnack}) implies that
\begin{equation}\label{e1}
\left|\frac{\nabla u_\alpha(x)}{u_\alpha(x)}\right|\leq C_2 d_g(x_\alpha,x)^{-1}
\end{equation}
in $\Omega_\alpha,$ and therefore
\begin{equation}\label{e2}
s_\alpha|\nabla\ln u_\alpha(x)|\leq 6C_2
\end{equation}
in $\Omega_\alpha$. Taking $C_3\geq 6C_2$, we get the first inequality from (\ref{e1}).
Then, from (\ref{e2}) and from the fact that the domain is an annulus  $\Omega_\alpha=B_{x_\alpha}(6s_\alpha)\backslash B_{x_\alpha}(\frac{1}{6}s_\alpha)$, we estimate that
$$\sup_{\Omega_\alpha}\ln u_\alpha- \inf_{\Omega_\alpha}\ln u_\alpha\leq l_\alpha(\Omega_\alpha)||\nabla \ln u_\alpha||_{L^\infty(\Omega_\alpha)}\leq 42C_2,$$
where $l_\alpha(\Omega_\alpha)$ is the infimum of the length of a curve in $\Omega_\alpha$ drawn between a maximum and a minimum of $u_\alpha$.  Equivalently
$$\sup_{\Omega_\alpha} u_\alpha\leq e^{42C_2}\inf_{\Omega_\alpha} u_\alpha,$$
so it suffices to take $C_3=e^{42C_2}$.
\end{proof}
\subsection{Local blow-up analysis}
In order to show that $u_\alpha$ is bounded in $\mathcal{C}^1$, we define a \textit{blow-up sequence} $(x_\alpha)_\alpha$ with $(\rho_\alpha)_\alpha$ as follows: let $(x_\alpha)_\alpha$ be critical points of $(u_\alpha)_\alpha$ and $(\rho_\alpha)_\alpha$ positive numbers such that they verify the following three conditions:
\begin{equation}\label{est: blow up no 1}
    0<\rho_\alpha<\frac{1}{7}i_g(M),
\end{equation}
\begin{equation}\label{est: blow up no 2}
    \rho_\alpha^{\frac{n-2}{2}}\sup_{B_{x_\alpha}(6\rho_\alpha)}u_\alpha\to\infty,
\end{equation}
and
\begin{equation}\label{est: blow up no 3}
    d_g(x_\alpha,x)^{\frac{n-2}{2}}\left[u_\alpha(x)+\left|\frac{\nabla u_\alpha(x)}{u_\alpha(x)}\right|^{\frac{n-2}{2}}\right]\leq C_2\quad\forall x\in B_{x_\alpha}(7\rho_\alpha),
\end{equation}
where $C_2$ is a constant. In the rest of the section, we denote 
$$\mu_\alpha^{1-\frac{n}{2}}:=u_\alpha(x_\alpha).$$
\begin{rem}
The limit as $\alpha\to \infty$ of a blow-up sequence is a concentration point, as seen from (\ref{est: blow up no 2}).
\end{rem}
\begin{rem}
Any sequence $(x_\alpha)_\alpha\subset (S_\alpha)_\alpha$ qualifies as a blow-up sequence as soon as (\ref{est: blow up no 2}) is verified. In this case, $(\rho_\alpha)_\alpha$ can be chosen as $$\rho_\alpha:=\min\left(\frac{1}{7}i_g(M),\frac{1}{2}\min_{1\leq i<j\leq N_\alpha}d_g(x_{i,\alpha},x_{j,\alpha})\right)$$ and $C_2=C_1$.
\end{rem}
Given any blow-up sequence, the following proposition gathers the central results of our local analysis for the reader's convenience: namely, it states the exact asymptotic profile of $(u_\alpha)_\alpha$ at distance $(\rho_\alpha)_\alpha$ of $(x_\alpha)_\alpha$ and it gives sharp pointwise asymptotic estimates on balls of radius $\rho_\alpha$. This is the result we shall point to whenever we want to describe the local asymptotic behaviour of $(u_\alpha)_\alpha$ around a concentration point corresponding to local maximum points $x_\alpha$. Note that in this case $\nabla u_\alpha(x_\alpha)=0$. 
\begin{prop}\label{prop}
Let $(x_\alpha)_\alpha$ and $(\rho_\alpha)_\alpha$ be a blow-up sequence. 
Then there exists $C_4>0$ such that
\begin{equation}
    u_\alpha(x)+d_g(x_\alpha,x)|\nabla u_\alpha(x)|\leq C_4\mu^{\frac{n-2}{2}}_\alpha d_g(x_\alpha,x)^{2-n}
\end{equation}
for all $x\in B_{x_\alpha}(6\rho_\alpha)\backslash\{x_\alpha\}.$ Moreover, we see that up to a subsequence, the asymptotic profile of $(u_\alpha)_\alpha$ is
\begin{equation}
    u_\alpha(x_\alpha)\rho^{n-2}_\alpha u_\alpha(\exp_\alpha(\rho_\alpha x))\to\frac{R_0^{n-2}}{|x|^{n-2}}+H(x)
\end{equation}
in $C^2_{loc}(B_0(5)\backslash\{0\})$, where $H$ is some harmonic function in $B_0(5)$ satisfying $H(0)=0.$ Here $R_0^2=\frac{n(n-2)}{f(x_0)}$ where $x_0=\lim_{\alpha\to\infty} x_\alpha$.
\end{prop}
The proof of this proposition is the subject of this section. It will follow from Lemma \ref{lem: former lemma 6} and Lemma \ref{lem: former lemma 7} below. We first describe the asymptotic profile at distance $(\mu_\alpha)_\alpha$ of $(x_\alpha)_\alpha$ as $\alpha\to\infty$ of any blow-up sequence.
\begin{lemma}\label{lem: former lemma 4}
Let $(x_\alpha)_\alpha$ with $(\rho_\alpha)_\alpha$ be a blow-up sequence. We have
\begin{equation}\label{est: estimate of old lemma 4}
    \mu_\alpha^{\frac{n-2}{2}}u_\alpha\left(\exp_{x_\alpha}(\mu_\alpha x)\right)\to\left(1+\frac{f(x_0)|x|^2}{n(n-2)}\right)^{1-\frac{n}{2}}
\end{equation}
in $\mathcal{C}^1_{loc}(\R^n)$ as $\alpha\to\infty$, with $\mu^{1-\frac{n}{2}}_\alpha:=u_\alpha(x_\alpha)$ and $x_0:=\lim_{\alpha\to\infty}x_\alpha$, up to a subsequence.
\end{lemma}

\begin{proof}[Proof of Lemma \ref{lem: former lemma 4}:]
The proof involves similar arguments to the ones used for Lemma $\ref{lem: former lemma 2}$. Let $y_\alpha\in B_{x_\alpha}(6\rho_\alpha)$ be such that
$$
    u_\alpha(y_\alpha)+\left|\frac{\nabla u_\alpha(y_\alpha)}{u_\alpha(y_\alpha)}\right|^{\frac{n-2}{2}}=\sup_{B_{x_\alpha}(6\rho_\alpha)}\left(u_\alpha(x)+\left|\frac{\nabla u_\alpha(x)}{u_\alpha(x)}\right|^{\frac{n-2}{2}}\right)
$$
and let
$$
    \nu_\alpha^{1-\frac{n}{2}}:=u_\alpha(y_\alpha)+\left|\frac{\nabla u_\alpha(y_\alpha)}{u_\alpha(y_\alpha)}\right|^{\frac{n-2}{2}}.
$$
Conditions (\ref{est: blow up no 2}) and (\ref{est: blow up no 3}) imply that
$$
    \frac{\rho_\alpha}{\nu_\alpha}\to \infty
$$
and
$$
    d_g(x_\alpha,y_\alpha)\leq C_2^{\frac{2}{n-2}}\nu_\alpha.
$$
It follows that the coordinates of $y_\alpha$ in the exponential chart around $x_\alpha$ defined as $\tilde{y}_\alpha:=\nu^{-1}\exp^{-1}_{x_\alpha}(y_\alpha)$ are bounded by $C_2^{\frac{2}{n-2}}$. Up to a subsequence, we may choose a finite limit $\tilde{y}_0:=\lim_{\alpha\to\infty}\tilde{y}_\alpha$. We denote 
$$v_\alpha(x)=\nu_\alpha^{\frac{n-2}{2}}u_\alpha\left(\exp_{x_\alpha}(\nu_\alpha x)\right)\quad\text{ and}\quad g_\alpha(x)=\left(\exp^*_{x_\alpha}g\right)\left(\exp_{x_\alpha}(\nu_\alpha x)\right)$$
for $\displaystyle x\in\Omega_\alpha:=B_0\left(\frac{\rho_\alpha}{\nu_\alpha}\right).$ As before, $g_\alpha\to\xi$ in $\mathcal{C}^2_{loc},$ $v_\alpha=O(1),$ and $\left|\frac{\nabla v_\alpha}{v_\alpha}\right|=O(1)$. By applying the same analysis as in the proof of Lemma $\ref{lem: former lemma 2}$, we get that, up to passing to a subsequence, there exists $U:=\lim_{\alpha\to\infty}v_\alpha$ in $\mathcal{C}^1_{loc}(\R^n)$, with $x_0:=\lim_{\alpha\to\infty}x_\alpha$,
$$
    \Delta U= f(x_0) U^{q-1}.
$$
where
$$U(x)=\left(1+\frac{f(x_0)|x-\tilde{y_0}|^2}{n(n-2)}\right)^{1-\frac{n}{2}}.$$
We know that $x_\alpha$ are local maxima for $u_\alpha$, so both $0$ and $\tilde{y}_0$ are maxima of $U$. However, since $U$ admits a unique maximum, we conclude that $\tilde{y}_0=0$.\end{proof}
We recall the aim of this section is to show that concentration points do not exist for the system (\ref{systemshort}). So far, we have obtained a pointwise estimate (\ref{est: initial estimate}) that holds everywhere on $M$ and an asymptotic profile in the neighbourhood of $x_\alpha$, a blow-up sequence; we aim to also find estimates around $x_\alpha$. We defined $(\rho_\alpha)_\alpha$ as the quantity describing the sphere of dominance of the blow-up sequence $(x_\alpha)_\alpha$. However, the influence of other blow-up sequences may be felt earlier. Let $\varphi_\alpha : (0, \rho_\alpha)\to \R^+$ be the average of $u_\alpha$ defined as
$$
    \varphi_\alpha(r) :=\frac{1}{|\partial B_{x_\alpha}(r)|_g}\int_{\partial B_{x_\alpha}(r)} u_\alpha  d\sigma_g.
$$
It follows from Lemma \ref{lem: former lemma 4} that
\begin{equation}\label{est: former lemma 4 but for averages}
    (\mu_\alpha r)^{\frac{n-2}{2}}\varphi_\alpha(\mu_\alpha r)\to r^{\frac{n-2}{2}}\left(1+\frac{f(x_0) r^2}{n(n-2)}\right)^{1-\frac{n}{2}}
\end{equation}
in $\mathcal{C}^1_{loc}([0,+\infty))$. We define
\begin{equation}\label{def: r alpha}
    r_\alpha:=\sup_{r\in(2R_0\mu_\alpha,\rho_\alpha)}\left\{s^{\frac{n-2}{2}}\varphi_\alpha(s)\text{ is non-increasing in } (2R_0\mu_\alpha, r)\right\}
\end{equation}
with 
$$R_0^2:=\frac{n(n-2)}{f(x_0)}.$$
Note that
$$
    \text{ if }r_\alpha<\rho_\alpha, \quad\text{ then } \left(r^{\frac{n-2}{2}}\varphi_\alpha(r)\right)'(r_\alpha)=0.
$$
Since $r^{\frac{n-2}{2}}U$ is non-increasing in $[2R_0,\infty)$, then (\ref{est: former lemma 4 but for averages}) implies
\begin{equation}\label{ess: between r alpha and mu alpha}
    \frac{r_\alpha}{\mu_\alpha}\to\infty.
\end{equation}
so $\mu_\alpha=o(r_\alpha)$.
\begin{rem}
Considering that the asymptotic profile of blow-up sequences $(x_\alpha)_\alpha$, which is a bump function with a unique maximum, the quantity $r_\alpha$ is an indicator of the beginning of the influence of neighbouring blow-up sequences within the sphere of dominance, as the average of $u_\alpha$ is no longer decreasing.
\end{rem}
Let 
\begin{equation}\label{def: eta}
\eta_\alpha:=\sup_{B_{x_\alpha}(6r_\alpha)\backslash B_{x_\alpha}(\frac{1}{6}r_\alpha)}u_\alpha.
\end{equation} 
Note that, by Lemma \ref{lem: former lemma 3},
$$\frac{1}{C_3}\sup_{B_{x_\alpha}(6s_\alpha)\backslash B_{x_\alpha}\left(\frac{1}{6}s_\alpha\right)}u_\alpha\leq\varphi_\alpha(s_\alpha)\leq C_3\inf_{B_{x_\alpha}(6s_\alpha)\backslash B_{x_\alpha}\left(\frac{1}{6}s_\alpha\right)}u_\alpha$$
for $0<s_\alpha\leq r_\alpha$ and all $\alpha$. By (\ref{est: former lemma 4 but for averages}), we obtain the estimate
\begin{equation}\label{est: using the harnack}
\lim_{R\to\infty}\limsup_{\alpha\to\infty}\sup_{B_{x_\alpha}(6r_\alpha)\backslash B_{x_\alpha}(R\mu_\alpha)}d_g(x_\alpha,x)^{\frac{n-2}{2}}u_\alpha=0.
\end{equation}
Thus,
\begin{equation}\label{est: rel size to max}
    r_\alpha^2\eta^{q-2}_\alpha\to 0
\end{equation}
as $\alpha\to\infty.$ It is important to note that this implies that
\begin{equation}\label{ess: ok for former lemma 5}
r_\alpha\to 0
\end{equation}
as $\alpha\to\infty$ since $u_\alpha\geq\varepsilon$ by (\ref{varepsilon}). We now prove a pointwise asymptotic estimate for $u_\alpha$ in $B_{x_\alpha}(6r_\alpha)\backslash\{x_\alpha\}.$
\begin{lemma}\label{lem: former lemma 5}
Let $(x_\alpha)_\alpha$ with $(\rho_\alpha)_\alpha$ be a blow-up sequence. Then, for any $0<\varepsilon<\frac{1}{2},$ there exists $C_\varepsilon>0$ such that
\begin{equation*}
    u_\alpha(x)\leq C_\varepsilon\left(\mu_\alpha^{\frac{n-2}{2}(1-2\varepsilon)}d_g(x_\alpha,x)^{(n-2)(1-\varepsilon)}+\eta_\alpha\left(\frac{r_\alpha}{d_g(x_\alpha,x)}\right)^{(n-2)\varepsilon}\right)
\end{equation*}
for all $x\in B_{x_\alpha}(6r_\alpha)\backslash\{x_\alpha\}.$
\end{lemma}
\begin{proof}[Proof of Lemma \ref{lem: former lemma 5}:]
Let $G$ be a Green function for the Laplace operator $\Delta_g$ on $M$ with $G> 0$. Recall the following estimates, that can be found in Aubin \cite{Au}:
\begin{equation}\label{ess: G on M}
\begin{array}{c}
    \left|d_g(x_\alpha,y)^{n-2}G(x,y)-\frac{1}{(n-2)\omega_{n-1}}\right|\leq\tau\left(d_g(x,y)\right)\\
        \left|d_g(x_\alpha,y)^{n-1}|\nabla G(x,y)|-\frac{1}{\omega_{n-1}}\right|\leq\tau\left(d_g(x,y)\right)
\end{array}
\end{equation}
where $\tau:\R^+\to \R^+$ is a continuous function satisfying $\tau(0)=0$.
For a fixed $\varepsilon$, let
\begin{equation}\label{def: uhm monster function i don't know it's late}
    \Phi^{\varepsilon}_\alpha(x):=\mu_\alpha^{\frac{n-2}{2}(1-2\varepsilon)}G(x_\alpha,x)^{1-\epsilon}+\eta_\alpha r_\alpha^{(n-2)\varepsilon}G(x_\alpha,x)^\varepsilon
\end{equation}
and let $y_\alpha\in\overline{B_{x_\alpha}(6 r_\alpha)}\backslash\{x_\alpha\}$ be such that
\begin{equation}\label{def of y}
    \sup_{B_{x_\alpha}(6r_\alpha)}\frac{u_\alpha}{\Phi_\alpha^{\varepsilon}}=\frac{u_\alpha(y_\alpha)}{\Phi_\alpha^\varepsilon(y_\alpha)},
\end{equation}
We continue by studying the following two cases, separately.
\par \textbf{First case:}
Assume that the relative size of $d_g(x_\alpha,y_\alpha)$ with respect to $\mu_\alpha$ is
\begin{equation}\label{ess: position of y}
    R:=\lim_{\alpha\to\infty}\frac{d_g(x_\alpha,y_\alpha)}{\mu_\alpha}\quad\text{with }R\in[0,\infty).
\end{equation}
Thanks to Lemma \ref{lem: former lemma 4},
\begin{equation*}
    \mu_\alpha^{\frac{n-2}{2}}u_\alpha(y_\alpha)=\left(1+\frac{R^2}{R_0^2}\right)^{1-\frac{n}{2}}+o(1),
\end{equation*}
so that, whenever $R\in[0,\infty)$, using (\ref{ess: between r alpha and mu alpha}), (\ref{ess: G on M}) and (\ref{ess: position of y}), it is easily shown that
$$\frac{u_\alpha(y_\alpha)}{\Phi_\alpha^\varepsilon(y_\alpha)}\to((n-2)\omega_{n-1})^{1-\varepsilon}R^{(n-2)(1-\varepsilon)}\left(1+\frac{R^2}{R_0^2}\right)^{1-\frac{n}{2}}$$
as $\alpha\to\infty.$
\par \textbf{Second case:} It remains to study the case 
$$\lim_{\alpha\to\infty}\frac{d_g(x_\alpha,y_\alpha)}{\mu_\alpha}\to\infty\quad\text{ as }\alpha\to\infty.$$ 
If $(y_\alpha)_\alpha$ sits on the outer boundary $\partial B_{x_\alpha}(6r_\alpha)$, then by (\ref{ess: ok for former lemma 5}), (\ref{ess: G on M}), and (\ref{def: uhm monster function i don't know it's late}),
$$\frac{u_\alpha(y_\alpha)}{\Phi^{\varepsilon}_\alpha(y_\alpha)}\leq \left(6^{n-2}(n-2)\omega_{n-1}\right)^\varepsilon + o(1).$$
Otherwise, if up to a subsequence $y_\alpha\in B_{x_\alpha}(6r_\alpha)$,
then
$$\frac{\Delta_{g}u_\alpha(y_\alpha)}{u_\alpha(y_\alpha)}\geq\frac{\Delta_g\Phi^\varepsilon_\alpha(y_\alpha)}{\Phi_\alpha^{\varepsilon}(y_\alpha)}$$
as a consequence of the fact that $y_\alpha$ is the maximum of $\displaystyle\frac{u_\alpha}{\Phi_\alpha^\varepsilon}$. On the other hand, taking note of the sign of the dominant gradient term in equations (\ref{eq: EL alpha}), we see that
\begin{equation}\label{because of the gradient sign}
\begin{array}{r l}
    \Delta_g u_\alpha&= -h_\alpha u_\alpha+f_\alpha u_\alpha^{q-1}+ \frac{a_\alpha}{u_\alpha^{q+1}}-\frac{b_\alpha}{u_\alpha}-\frac{\langle\nabla u_\alpha, Y_\alpha\rangle^2}{u_\alpha^{q+3}}\\
     &\quad-c_\alpha\langle\nabla u_\alpha,Y_\alpha\rangle\left[\frac{d_\alpha}{u_\alpha^2}+\frac{1}{u_\alpha^{q+2}}\right]\\
     &\leq C u_\alpha^{q-1},
\end{array}
\end{equation}
where $C$ is a constant depending on $\theta$ and $T$. Here we used (\ref{varepsilon}). Finally, thanks to (\ref{est: using the harnack}),
$$d_g(x_\alpha,y_\alpha)^2\frac{\Delta_g u_\alpha(y_\alpha)}{u_\alpha(y_\alpha)}\leq C d_g(x_\alpha,y_\alpha)^2 u^{q-2}_\alpha(y_\alpha)\to 0.$$
To conclude, (\ref{ess: ok for former lemma 5}) and (\ref{ess: G on M}) imply that
$$d_g(x_\alpha,y_\alpha)^2\frac{\Delta_g \Phi^\varepsilon_\alpha(y_\alpha)}{\Phi^\varepsilon_\alpha(y_\alpha)}=\varepsilon(1-\varepsilon)(n-2)^2+o(1).$$
We deduce that $u_\alpha(y_\alpha)=O(\Phi^\varepsilon_\alpha(y_\alpha)).$ The study of the previous two cases ends the proof of the lemma.
\end{proof}
The following lemma improves the estimate we've just obtained and gives a very important bound on the size of $r_\alpha$.
\begin{lemma}\label{lem: former lemma 6}
Let $(x_\alpha)_\alpha$ with $(\rho_\alpha)_\alpha$ be a blow-up sequence. Then there exists $C_4>0$ such that
\begin{equation}\label{est: for former lemma 6}
    u_\alpha(x)+d_g(x_\alpha,x)|\nabla u_\alpha(x)|\leq C_4\mu^{\frac{n-2}{2}}_\alpha \left(d_g(x_\alpha,x)+\mu_\alpha\right)^{2-n}
\end{equation}
for all $x\in B_{x_\alpha}(6r_\alpha)\backslash\{x_\alpha\}.$ Moreover, $r_\alpha^2=O(\mu_\alpha)$.
\end{lemma}
\begin{proof}[Proof of Lemma \ref{lem: former lemma 6}:]
It suffices to prove the estimate for $u_\alpha$; the rest follows as an immediate consequence of Lemma \ref{lem: former lemma 3}. We start by showing that for any sequence $z_\alpha\in \overline{B_{x_\alpha}(6r_\alpha)}\backslash\{x_\alpha\}$, there holds
\begin{equation}\label{est: form lem 6 most of the work}
u_\alpha(z_\alpha)=O\left(\mu_\alpha^{\frac{n-2}{2}}d_g(x_\alpha,z_\alpha)^{2-n}+\eta_\alpha\right).
\end{equation}
First, if $d_g(x_\alpha,z_\alpha)=O(\mu_\alpha)$, it falls within the range described in Lemma \ref{lem: former lemma 4}. On the other hand, when $r_\alpha=O\left(d_g(x_\alpha,z_\alpha)\right),$ we use Lemma \ref{lem: former lemma 3} together with (\ref{def: eta}). It remains to consider the intermediary case:
$$
    \frac{d_g(x_\alpha,z_\alpha)}{\mu_\alpha}\to\infty\quad\text{ and }\quad\frac{d_g(x_\alpha,z_\alpha)}{r_\alpha}\to 0\quad\text{ at }\alpha\to\infty.
$$
According to the Green representation formula, 
$$
    u_\alpha(z_\alpha)=O\left(\int_{B_{x_\alpha}(6r_\alpha)}d_g(z_\alpha,x)^{2-n}\Delta_g u_\alpha(x)\,dv_g\right)+O(\eta_\alpha),
$$
where the second term corresponds to the boundary element. Recall that
$$\Delta_g u_\alpha\leq C u_\alpha^{q-1}$$
because of the sign of the dominant gradient term, see (\ref{because of the gradient sign}). Using (\ref{varepsilon}), (\ref{est: rel size to max}) and Lemma \ref{lem: former lemma 5}, we can write that
$$
\begin{array}{l}
     \int_{B_{x_\alpha}(6r_\alpha)}d_g(z_\alpha,x)^{2-n}u_\alpha^{q-1}(x)\,dv_g\\ 
    \displaystyle =O\left(\mu_\alpha^{\frac{n}{2}-1}d_g(x_\alpha,z_\alpha)^{2-n}\right)\\ 
     +O\left(\mu_\alpha^{\frac{n+2}{2}(1-2\varepsilon)}\int_{B_{x_\alpha}(6r_\alpha)\backslash B_{x_\alpha}(\mu_\alpha)}d_g(z_\alpha,x)^{2-n}d_g(x_\alpha,x)^{-(n+2)(1-\varepsilon)}\,dv_g\right)\\ 
     +O\left(\eta_\alpha^{q-1}r_\alpha^{(n+2)\varepsilon}\int_{B_{x_\alpha}(6r_\alpha)\backslash B_{x_\alpha}(\mu_\alpha)}d_g(z_\alpha,x)^{2-n}d_g(x_\alpha,x)^{-(n+2)\varepsilon}\,dv_g\right)\\ 
     = O\left(\mu^{\frac{n}{2}-1}_\alpha d_g(x_\alpha,z_\alpha)^{2-n}\right)+O\left(\eta_\alpha^{q-1}r_\alpha^2\right)\\ 
     = O\left(\mu^{\frac{n}{2}-1}_\alpha d_g(x_\alpha,z_\alpha)^{2-n}\right)+O\left(\eta_\alpha\right).
\end{array}
$$
In order to get estimate (\ref{est: for former lemma 6}), it suffices to show that
\begin{equation}\label{est: of eta}
\eta_\alpha=O\left(\mu_\alpha^{\frac{n-2}{2}}r_\alpha^{2-n}\right).
\end{equation}
For any fixed $0<\delta<1$, taking $\alpha$ large enough, then the monotonicity of $r^{\frac{n-2}{2}}\varphi_\alpha(r)$ expressed in the definition of $r_\alpha$, see (\ref{def: r alpha}), and the fact that $\mu_\alpha=o(r_\alpha)$, see (\ref{ess: between r alpha and mu alpha}), imply that
$$r_\alpha^{\frac{n-2}{2}}\varphi_\alpha(r_\alpha)\leq (\delta r_\alpha)^{\frac{n-2}{2}}\varphi_\alpha(\delta r_\alpha)\quad\text{ for all }0<\delta<1,$$
so by Lemma \ref{lem: former lemma 3}, 
$$\frac{1}{C_3}\eta_\alpha\leq \delta^{\frac{n-2}{2}}\sup_{\partial B_{x_\alpha}(\delta r_\alpha)}u_\alpha.$$
According to estimate (\ref{est: form lem 6 most of the work}), this leads to
 $$\eta_\alpha\leq C\left(\mu_\alpha^{\frac{n-2}{2}}\delta^{2-n}r_\alpha^{2-n}+\eta_\alpha\right)\delta^{\frac{n-2}{2}},$$
where $C$ is independent of $\delta$ and $\alpha$. Choosing $\delta$ small enough leads to (\ref{est: of eta}).
\par Estimates (\ref{varepsilon}) and (\ref{est: of eta}) imply that
\begin{equation}\label{size of r alpha}
r_\alpha^2= O(\mu_\alpha).
\end{equation}
This ends the proof of the lemma.
\end{proof}

Given a blow-up sequence $(x_\alpha)_\alpha$ with $(\rho_\alpha)_\alpha$, the following lemma gives the exact asymptotic profile of $(u_\alpha)_\alpha$ at distance $(\rho_\alpha)_\alpha$ of $(x_\alpha)_\alpha$. 
\begin{lemma}\label{lem: former lemma 7}
Let $(x_\alpha)_\alpha$ and $(\rho_\alpha)_\alpha$ be a blow-up sequence. Then we have that $r_\alpha=\rho_\alpha$, where $r_\alpha$ is as in (\ref{def: r alpha}). Up to a subsequence, we have
\begin{equation}\label{est final profile in r alpha}
    u_\alpha(x_\alpha)\rho^{n-2}_\alpha u_\alpha(\exp_\alpha(\rho_\alpha x))\to\frac{R_0^{n-2}}{|x|^{n-2}}+H(x)
\end{equation}
in $C^2_{loc}(B_0(5)\backslash\{0\})$, where $H$ is some harmonic function in $B_0(5)$ satisfying $H(0)=0.$
\end{lemma}
\begin{proof}[Proof of Lemma \ref{lem: former lemma 7}:]
First, we prove that, up to a subsequence,
\begin{equation}\label{prefinal profile}
u_\alpha(x_\alpha)r^{n-2}_\alpha u_\alpha(\exp_\alpha(r_\alpha x))\to\frac{R_0^{n-2}}{|x|^{n-2}}+H(x).
\end{equation}
Let us define the following rescaled quantities:
$$
    \hat{u}_\alpha(x)=\mu_\alpha^{1-\frac{n}{2}}r_\alpha^{n-2}u_\alpha\left(\exp_{x_\alpha}(r_\alpha x)\right)\text{ and }\hat{g}_\alpha(x)=\exp_{x_\alpha}^*g\left(\exp_{x_\alpha}(r_\alpha x)\right).
$$
Then
$$
    \Delta_{\hat{g}_\alpha}\hat{u}_\alpha=\hat{F}_\alpha,
$$
in $B_0(\delta r_\alpha^{-1})$ for some $\delta>0$ small enough, with
\begin{equation}\label{eq: form lemma 7's hat F}
\begin{array}{r l}
     \hat{F}_\alpha=&    \displaystyle -\mu_\alpha^{1-\frac{n}{2}}r_\alpha^{n}h_\alpha\left(\exp_{x_\alpha}(r_\alpha x)\right)u_\alpha\left(\exp_{x_\alpha}(r_\alpha x)\right) \\ 
    &  +\mu_\alpha^{1-\frac{n}{2}}r_\alpha^{n}f_\alpha \left(\exp_{x_\alpha}(r_\alpha x)\right)u_\alpha^{q-1}\left(\exp_{x_\alpha}(r_\alpha x)\right) \\ 
    &     +\mu_\alpha^{1-\frac{n}{2}}r_\alpha^{n}\frac{a_\alpha\left(\exp_{x_\alpha}(r_\alpha x)\right)}{u^{q+1}_\alpha\left(\exp_{x_\alpha}(r_\alpha x)\right)} -\mu_\alpha^{1-\frac{n}{2}}r_\alpha^n\frac{b_\alpha\left(\exp_{x_\alpha}(r_\alpha x)\right)}{u_\alpha \left(\exp_{x_\alpha}(r_\alpha x)\right)}\\ 
    &     -\mu_\alpha^{1-\frac{n}{2}}r_\alpha^{n} \frac{\langle\nabla u_\alpha\left(\exp_{x_\alpha}(r_\alpha x)\right), Y_\alpha\left(\exp_{x_\alpha}(r_\alpha x)\right)\rangle^2}{u_\alpha^{q+3}\left(\exp_{x_\alpha}(r_\alpha x)\right)}\\ 
    &  -\mu_\alpha^{1-\frac{n}{2}}r_\alpha^n c_\alpha\left(\exp_{x_\alpha}(r_\alpha x)\right)\langle\nabla u_\alpha,Y_\alpha\rangle\left(\exp_{x_\alpha}(r_\alpha x)\right) \Big[ \frac{d_\alpha\left(\exp_{x_\alpha}(r_\alpha x)\right)}{u_\alpha\left(\exp_{x_\alpha}(r_\alpha x)\right)^2}\\ 
    & + \frac{1}{u_\alpha\left(\exp_{x_\alpha}(r_\alpha x)\right)^{q+2}}\Big]\\ 
\end{array}
\end{equation}
Thanks to Lemma \ref{lem: former lemma 6}, we know that $|\hat{u}_\alpha|\leq C_K$  and $|\hat{F}_\alpha|=O(1)$ on any compact $K\subset B_0(5)\backslash\{0\}$. By standard elliptic theory, up to a subsequence,
$$
    \hat{u}_\alpha\to\hat{U}\text{ in }\mathcal{C}^1_{loc}(B_0(5)\backslash\{0\})
$$
with
$$\Delta_\xi\hat{U}=0\text{ in }B_0(5)\backslash\{0\}.$$
Separate $\hat{U}$ into the sum of a regular harmonic function and a singular part
$$\hat{U}=\frac{\lambda}{|x|^{n-2}}+H(x),$$
where $\lambda\geq 0.$ 
\par To get (\ref{prefinal profile}), it remains to show that $\lambda=R_0^{n-2}.$ For any $\delta>0$:
\begin{equation}\label{eq: lemma 7 minutia}
    \int_{B_0(\delta)}\hat{F}_\alpha\,dv_{\hat{g}_\alpha}=-\int_{\partial B_0(\delta)}\partial_\nu\hat{u}_\alpha\,d\sigma_{\hat{g}_\alpha}.
\end{equation}
Using the equation (\ref{eq: form lemma 7's hat F}), we estimate the left hand side of (\ref{eq: lemma 7 minutia}). In particular, 
$$
\begin{array}{c}
\displaystyle \int_{B_0(\delta)}f_\alpha\left(\exp_{x_\alpha}(r_\alpha x)\right)\mu_\alpha^{1-\frac{n}{2}}r_\alpha^{n}u_\alpha^{q-1}\left(\exp_{x_\alpha}(r_\alpha x)\right)\,dx \\ \\
\displaystyle =\int_{B_0(\delta\frac{r_\alpha}{\mu_\alpha})}f_\alpha\left(\exp_{x_\alpha}(\mu_\alpha z)\right)\mu_\alpha^{\frac{n+2}{2}}u_\alpha^{q-1}\left(\exp_{x_\alpha}(\mu_\alpha z)\right)\,dz
\end{array}
$$
where $\displaystyle z=\frac{r_\alpha}{\mu_\alpha}x$, and by Lemma \ref{lem: former lemma 4} and Lemma \ref{lem: former lemma 6},
$$
\begin{array}{c}
   \displaystyle \lim_{\alpha\to\infty}\int_{B_0(\delta)}f_\alpha\left(\exp_{x_\alpha}(r_\alpha x)\right)\mu_\alpha^{1-\frac{n}{2}}r_\alpha^{n}u_\alpha^{q-1}\left(\exp_{x_\alpha}(r_\alpha x)\right)\,dx \\ \\
   \displaystyle =f(x_0)\int_{R^n}\left(1+\frac{|x|^2}{R_0^2}\right)^{-1-\frac{n}{2}}\,dx.
\end{array}
$$
The gradient terms are controlled with the estimate (\ref{est: initial estimate}), and together with (\ref{varepsilon}), we obtain that the dominant gradient term of (\ref{eq: form lemma 7's hat F}) verifies
\begin{equation}\label{es 1}
    \begin{array}{r l}
  \displaystyle  \int_{B_0(\delta)}\frac{|\nabla u_\alpha|^2}{u_\alpha^{q+3}}\left(\exp_{x_\alpha}(r_\alpha x)\right)\mu_\alpha^{1-\frac{n}{2}}r_\alpha^n\,dx\leq  
    & \displaystyle  C \mu_\alpha^{1-\frac{n}{2}}r_\alpha^{n-2}\int_{B_0(\delta)}|x|^{-2}\,dx  \\ \\
    & \leq \displaystyle C\omega_{n-1} \mu_\alpha^{1-\frac{n}{2}}r_\alpha^{n-2}\delta^{n-2}.
\end{array}
\end{equation}
As $\mu_\alpha^{1-\frac{n}{2}}r_\alpha^{n-2}=O(1),$ the integral does not vanish as $\alpha\to\infty$; its size depends on $\delta$. The remaining terms in (\ref{eq: form lemma 7's hat F}) are negligible. Thus
$$
    \int_{B_0(\delta)}\hat{F}_\alpha\, dv_{\hat{g}_\alpha}= f(x_0)\int_{R^n}\left(1+\frac{|x|^2}{R_0^2}\right)^{-1-\frac{n}{2}}\,dx+o(1)+O(\delta^{n-2})
$$
for any $\delta>0$. It follows that
$$f(x_0)\int_{\R^n}\left(1+\frac{|x|^2}{R_0^2}\right)^{-1-\frac{n}{2}}\,dx=(n-2)\omega_{n-1}R_0^{n-2}.$$ Note also that the right hand side of (\ref{eq: lemma 7 minutia}) verifies
$$\begin{array}{r l}
-\int_{\partial B_0(\delta)}\partial_\nu \hat{u}_\alpha\,d\sigma_{\check{g}_\alpha}&=-\int_{\partial B_0(\delta)}\partial_\nu\hat{U}+o(1)\\
&=\lambda(n-2)\omega_{n-1}+o(1).
\end{array}
$$
since $H$ is smooth and harmonic.
Since
$$
    \lambda(n-2)\omega_{n-1}=R^{n-2}_0(n-2)\omega_{n-1}+O(\delta^{n-2})+o(1),
$$
for any $\delta>0$, we get that $\lambda=R_0^{n-2}.$
\par Finally, let us prove that $H(0)=0.$ The equation's dominant terms are invariant by rescaling, which leads us to use a Pohozaev identity to obtain new estimates for the remaining terms. Let $\Omega_\alpha$ correspond to $B_0(\delta r_\alpha)$ in the exponential chart at $x_\alpha\in M$ and let $X_\alpha=\frac{1}{2}\nabla d_g(x_\alpha,x)^2$ be the vector field of coordinates. Using integration by parts,
$$
\begin{array}{r l}
 \int_{\Omega_\alpha}\nabla u_\alpha(X_\alpha)\Delta_gu_\alpha\,dv_g=&  \int_{\Omega_\alpha}\langle\nabla \left(\nabla u_\alpha(X_\alpha)\right),\nabla u_\alpha\rangle\, dv_g\\ 
    & - \int_{\partial \Omega_\alpha}\nabla u_\alpha(X_\alpha)\partial_{\nu}u_\alpha\,d\sigma_g\\ 
	 =
	& \int_{\Omega_\alpha}\nabla^\#\nabla u_\alpha( X_\alpha, \nabla u_\alpha) + \nabla^\# X_\alpha\left(\nabla u_\alpha, \nabla u_\alpha\right)\,dv_g\\ 
	 & - \int_{\partial \Omega_\alpha}\nabla u_\alpha(X_\alpha)\partial_{\nu}u_\alpha\,d\sigma_g,
\end{array}
$$
where $(\nabla^\# X_\alpha)=(\nabla^i X_\alpha)^j$. Since
$$
\begin{array}{c}

\int_{\Omega_\alpha}|\nabla u_\alpha|^2div_g X_\alpha\,dv_g+\int_{\Omega_\alpha}\langle\nabla\left(|\nabla u_\alpha|^2\right),X_\alpha\rangle\,dv_g\\

=\int_{\partial \Omega_\alpha}|\nabla u_\alpha|^2 \langle X_\alpha, \nu\rangle \,d\sigma_g,
\end{array}
$$
we can write that
\begin{equation}\label{eq: Pohozaev}
\begin{array}{r l}
\int_{\Omega_\alpha}
	&\displaystyle\left(\nabla u_\alpha(X_\alpha)+\frac{n-2}{2}u_\alpha\right)\Delta_g u_\alpha\,dv_g \\ 
	& =\int_{\Omega_\alpha}\left(\nabla^\# X_\alpha\left(\nabla u_\alpha, \nabla u_\alpha\right)-\frac{1}{2}\left(div_g X_\alpha\right)|\nabla u_\alpha|^2\right)\,dv_g\\ 
	&  + \int_{\partial\Omega_g}\left(\frac{1}{2}\langle X_\alpha,\nu\rangle |\nabla u_\alpha|^2-\nabla u_\alpha(X_\alpha)\partial_\nu u_\alpha-\frac{n-2}{2}u_\alpha\partial_\nu u_\alpha\right)\,d\sigma_g.
\end{array}
\end{equation}
We begin by analyzing the right-hand side of (\ref{eq: Pohozaev}). By our choice of $X_\alpha$, $(\nabla^\# X_\alpha)^{ij}=g^{ij}+O(d_g(x_\alpha,x)^2),$ and consequently
$$
\begin{array}{rl}
     \int_{\Omega_\alpha} &\displaystyle \left(\nabla^\# X_\alpha\left(\nabla u_\alpha, \nabla u_\alpha\right)-\frac{1}{2}\left(div_g X_\alpha\right)|\nabla u_\alpha|^2\right)\,dv_g  \\ 
     &  =O\left(\int_{\Omega_\alpha}d_g(x_\alpha,x)^2|\nabla u_\alpha|^2dv_g\right).
\end{array}
$$
According to (\ref{est: for former lemma 6}),
$$
\begin{array}{c}
\displaystyle  \int_{\Omega_\alpha}d_g(x_\alpha,x)^2|\nabla u_\alpha|^2\,dv_g\leq C\int_{\Omega_\alpha}\mu_\alpha^{n-2}\left(d_g(x_\alpha,x)+\mu_\alpha\right)^{4-2n}\,dv_g,
\end{array}
$$
so
$$\displaystyle \int_{\Omega_\alpha}d_g(x_\alpha,x)^2|\nabla u_\alpha|^2\,dv_g \leq\left\{  \begin{array}{l c}
      \displaystyle O(\mu_\alpha r_\alpha) & \displaystyle \text{ if }n=3 \\ \\
      
     \displaystyle O\left(\mu_\alpha^2\ln\frac{1}{\mu_\alpha}\right) & \displaystyle \text{ if }n=4 \\ \\
      \displaystyle O\left(\mu_\alpha^2\right) & \displaystyle \text{ if }n=5
\end{array}\right.$$
In all these three cases, thanks to (\ref{size of r alpha}), the integral is of the order $o(\mu_\alpha^{n-2} r_\alpha^{2-n})$.
From (\ref{est final profile in r alpha}),
$$
    \begin{array}{r l}
    \displaystyle \int_{\partial\Omega_\alpha} & \displaystyle  \left(\frac{1}{2}\langle X_\alpha,\nu\rangle |\nabla u_\alpha|^2-\nabla u_\alpha(X_\alpha)\partial_\nu u_\alpha-\frac{n-2}{2}u_\alpha\partial_\nu u_\alpha\right)\,d\sigma_g  \\ 
     &=\left(\frac{(n-2)^2}{2}\omega_{n-1}R_0^{n-2}H(0)+o(1)\right)\mu_\alpha^{n-2}r_\alpha^{2-n}.
    \end{array}
$$
Note that the boundary term does not depend on $\delta$, and as a result
\begin{equation}\label{eq: pohozaev id 1 in lem 7 former}
    \begin{array}{r l}
         \int_{\Omega_\alpha}&\displaystyle\left(\nabla u_\alpha(X_\alpha)+\frac{n-2}{2}u_\alpha\right)\Delta_g u_\alpha\,dv_g  \\ 
         & =\left(\frac{(n-2)^2}{2}\omega_{n-1}R_0^{n-2}H(0)+o(1)\right)\mu_\alpha^{n-2}r_\alpha^{2-n}
    \end{array}
\end{equation}
We now analyse the right hand side of (\ref{eq: Pohozaev}) by using (\ref{eq: EL alpha}):
\begin{equation}\label{eq: from the eq}
\begin{array}{r l}
    \int_{\Omega_\alpha}
    &\left(\nabla u_\alpha(X_\alpha)+\frac{n-2}{2}u_\alpha\right)\Delta_g u_\alpha\,dv_g  \\ 
    =& \int_{\Omega_\alpha}\left(\nabla u_\alpha(X_\alpha)+\frac{n-2}{2}u_\alpha\right)f_\alpha u_\alpha^{q-1}\,dv_g\\
    & -  \int_{\Omega_\alpha}\left(\nabla u_\alpha(X_\alpha)+\frac{n-2}{2}u_\alpha\right)\langle\nabla u_\alpha, Y_\alpha\rangle^2 u_\alpha^{-q-3}\,dv_g\\ 
    & -  \int_{\Omega_\alpha}\left(\nabla u_\alpha(X_\alpha)+\frac{n-2}{2}u_\alpha\right)c_\alpha\langle\nabla u_\alpha, Y_\alpha\rangle(d_\alpha u_\alpha^{-2}+u_\alpha^{-q-2})\,dv_g\\ 
    & +  \int_{\Omega_\alpha}\left(\nabla u_\alpha(X_\alpha)+\frac{n-2}{2}u_\alpha\right)\left(a_\alpha u_\alpha^{-q-1}-b_\alpha u_\alpha^{-1}- h_\alpha u_\alpha\right)\,dv_g
\end{array}
\end{equation}
and we look at each term in turn. By the estimates (\ref{varepsilon}) and (\ref{est: initial estimate}), we get
\begin{equation}\label{for dominant nabla term}
\begin{array}{r l}
 \left|\int_{\Omega_\alpha}\left(\nabla u_\alpha(X_\alpha)+\frac{n-2}{2}u_\alpha\right)\langle\nabla u_\alpha,Y_\alpha\rangle^2 u_\alpha^{-q-3}\,dv_g \right|
& \leq C\int_{\Omega_\alpha}d_g(x_\alpha,x)^{-2}dv_g\\ 
& \leq C(\delta r_\alpha)^{n-2}
\end{array}
\end{equation}
and, similarly,
\begin{equation}\label{for nondominant nabla term}
\begin{array}{c}
\left|\int_{\Omega_\alpha}\left(\nabla u_\alpha(X_\alpha)+\frac{n-2}{2}u_\alpha\right)c_\alpha\langle\nabla u_\alpha, Y_\alpha\rangle(d_\alpha u_\alpha^{-2}+u_\alpha^{-q-2})\,dv_g\right|\\
\leq C(\delta r_\alpha)^{n-1}.
\end{array}
\end{equation}
We also have that
\begin{equation}\label{for a and b}
\left|\int_{\Omega_\alpha}\left(\nabla u_\alpha(X_\alpha)+\frac{n-2}{2}u_\alpha\right)\left(a_\alpha u_\alpha^{-q-1}-b_\alpha u_\alpha^{-1}\right)\right|\leq Cr_\alpha^n.
\end{equation}
From (\ref{est: for former lemma 6}), we obtain that
\begin{equation}\label{for h}
\begin{array}{c}
\left|\int_{\Omega_\alpha}h_\alpha\left(\nabla u_\alpha(X_\alpha)+\frac{n-2}{2}u_\alpha\right)u_\alpha\,dv_g\right|\\=O\left(\int_{\Omega_\alpha}\mu_\alpha^{n-2}(\mu_\alpha+d_g(x_\alpha,x))^{4-2n}\,dx\right)\\
=o(\mu_\alpha^{n-2}r_\alpha^{2-n})
\end{array}
\end{equation}
for $3\leq n\leq 5.$ Using integration by parts,
\begin{equation}
\begin{array}{r l}
 \int_{\Omega_\alpha}\nabla u_\alpha(X_\alpha)f_\alpha u_\alpha^{q-1}\,dv_g=&\frac{1}{q} \int_{\partial \Omega_\alpha}f_\alpha r_\alpha u_\alpha^q\,d\sigma_g\\ 
& -\frac{1}{q}\int_{\Omega_\alpha}div_g X_\alpha f_\alpha u_\alpha^q\, dv_g \\ &- \frac{1}{q}\int_{\Omega_\alpha}\nabla f_\alpha(X_\alpha)u_\alpha^q\,dv_g.
\end{array}
\end{equation}
Thus we can write that
$$
\begin{array}{r l}
\int_{\Omega_\alpha}\left(\nabla u_\alpha(X_\alpha)+\frac{n-2}{2}u_\alpha\right)f_\alpha u_\alpha^{q-1}\,dv_g&=\frac{1}{q}r_\alpha\int_{\partial \Omega_\alpha}f_\alpha u_\alpha^q\,d\sigma_g\\
&\,+\int_{\Omega_\alpha}\left(-\frac{1}{q}div_g(X_\alpha)+\frac{n-2}{2}\right)f_\alpha u_\alpha^q\,dv_g\\
&\,-\frac{1}{q}\int_{\Omega_\alpha}\nabla f_\alpha(X_\alpha)u_\alpha^q\,dv_g.
\end{array}
$$
Since $div_g(X_\alpha)=n+O\left(d_g(x_\alpha,x)^2\right)$, this leads to
$$
\begin{array}{r l}
\int_{\Omega_\alpha}\left(\nabla u_\alpha(X_\alpha)+\frac{n-2}{2}u_\alpha\right)f_\alpha u_\alpha^{q-1}\,dv_g&=\frac{1}{q}r_\alpha\int_{\partial \Omega_\alpha}f_\alpha u_\alpha^q\,d\sigma_g\\
&\,+O\left(\int_{\Omega_\alpha} d_g(x_\alpha,x)^2 u_\alpha^q\,dv_g\right)\\
&\,-\frac{1}{q}\int_{\Omega_\alpha}\nabla f_\alpha(X_\alpha)u_\alpha^q\,dv_g.

\end{array}
$$
Using lemmas \ref{lem: former lemma 4} and \ref{lem: former lemma 6}, this leads to
$$
\begin{array}{r l}
\int_{\Omega_\alpha}\left(\nabla u_\alpha(X_\alpha)+\frac{n-2}{2}u_\alpha\right)f_\alpha u_\alpha^{q-1}\,dv_g&= O\left(\mu_\alpha^n r_\alpha^{-n}\right)+O(\mu_\alpha^2)\\
&-\frac{1}{q}\int_{\Omega_\alpha}\nabla f_\alpha(X_\alpha)u_\alpha ^q\,dv_g
\end{array}
$$
so that, thanks to (\ref{size of r alpha}),
\begin{equation}\label{for f}
\begin{array}{c}
\int_{\Omega_\alpha}\left(\nabla u_\alpha(X_\alpha)+\frac{n-2}{2}u_\alpha\right)f_\alpha u_\alpha^{q-1}\,dv_g=o\left(\mu_\alpha^{n-2}r_\alpha^{2-n}\right)\\-\frac{1}{q}\int_{\Omega_\alpha}\nabla f_\alpha(X_\alpha)u_\alpha^q\,dv_g
\end{array}
\end{equation}
if $3\leq n\leq 5.$
\par We claim that
\begin{equation}\label{for nabla f}
\int_{\Omega_\alpha}\nabla f_\alpha (X_\alpha)u_\alpha^q\,dv_g=o\left(\mu_\alpha^{n-2}r_\alpha^{2-n}\right).
\end{equation}
Thanks to (\ref{eq: pohozaev id 1 in lem 7 former}), (\ref{for dominant nabla term}), (\ref{for nondominant nabla term}), (\ref{for a and b}), (\ref{for h}), (\ref{for f}) and (\ref{for nabla f}), we see that
\begin{equation}\label{Hdelta}
H(0)=o(1)+\delta^4
\end{equation}
for any $\delta>0$, so by taking $\delta\to 0$ wee see that $H(0)=0$.
\par In order to prove (\ref{for nabla f}), we can first use Lemma \ref{lem: former lemma 6} to write that
$$\int_{\Omega_\alpha}\nabla f_\alpha(X_\alpha)u_\alpha^q\,dv_g=O(\mu_\alpha)$$
which leads to (\ref{for nabla f}) if $n=3,4$ thanks to (\ref{size of r alpha}), but is not enough for $n=5$. In order to improve the estimate in the case $n=5$, note that
$$\begin{array}{r l}
\displaystyle \int_{\Omega_\alpha}\nabla f_\alpha(X_\alpha)u_\alpha^q\,dv_g 
&\displaystyle =\partial_i f(x_\alpha)\int_{\Omega_\alpha}x^iu_\alpha^q\,dv_g\\ \\
&\displaystyle\quad +O\left(\int_{\Omega_\alpha}d_g(x_\alpha,x)^{1+\eta}u_\alpha^q\,dv_g\right)\\ \\
&\displaystyle =o\left(\mu_\alpha|\nabla f_\alpha(x_\alpha)|\right)+ O\left(\mu_\alpha^{1+\eta}\right)\\ \\
&\displaystyle =o\left(\mu_\alpha|\nabla f_\alpha(x_\alpha)|\right)+ o\left(\mu_\alpha^{3}r_\alpha^{-3}\right).
\end{array}$$
with $\eta>\frac{1}{2}.$ Thus it remains to prove that 
\begin{equation}\label{est: gjgjhvj}
|\nabla f_\alpha(x_\alpha)|=O\left(\mu_\alpha^2 r_\alpha^{-3}\right).
\end{equation}
As before, we use a Pohozaev-type identity. We make use of the equation's symmetry by translation, with $Z=Z^i$ a constant vector field in the exponential chart of $x_\alpha$. We can write that 
\begin{equation}\label{newpoz}
\int_{\Omega_\alpha}\nabla u_\alpha(Z_\alpha)\Delta_g u_\alpha\,dv_g=O\left(\int_{\Omega_\alpha}d_g(x_\alpha,x)|\nabla u_\alpha|^2\,dv_g +\int_{\partial \Omega_\alpha}|\nabla u_\alpha|^2\,d\sigma_g\right),
\end{equation}
which is $o(\mu_\alpha^2 r_\alpha^{-3})$. On the left-hand side, we use (\ref{eq: EL alpha}). Lemma \ref{lem: former lemma 3} and (\ref{varepsilon}) imply that
\begin{equation}\label{newpoz1}
\begin{array}{r l}
\displaystyle \int_{\Omega_\alpha}\nabla u_\alpha(Z_\alpha)\frac{\langle\nabla u_\alpha,Y_\alpha\rangle^2}{u_\alpha^{q+3}}
\,dv_g &\displaystyle \leq \frac{1}{\varepsilon^q}\int_{\Omega_\alpha}\left|\frac{\nabla u_\alpha}{u_\alpha}\right|^3\,dv_g\\ \\
&\displaystyle
\leq \frac{1}{\varepsilon^q}\int_{\Omega_\alpha}d_g(x_\alpha,x)^{-3}\,dv_g\\ \\
&\displaystyle =O(r^2_\alpha)=o(\mu_\alpha^2r_\alpha^{-3}).
\end{array}
\end{equation}
We see that
\begin{equation}\label{newpoz2}
\int_{\Omega_\alpha}\nabla u_\alpha(Z_\alpha)c_\alpha\langle\nabla u_\alpha, Y_\alpha\rangle\left(\frac{d_\alpha}{u_\alpha^2}+\frac{1}{u_\alpha^{q+2}}\right)
\,dv_g=O(r^2_\alpha)=o(\mu_\alpha^2r_\alpha^{-3}),
\end{equation}
and that the same holds for the terms corresponding to $h_\alpha$, $b_\alpha$ and $c_\alpha$. So (\ref{newpoz}), (\ref{newpoz1}) and (\ref{newpoz2}) imply that
$$ \int_{\Omega_\alpha}\nabla u_\alpha(Z_\alpha)f_\alpha u_\alpha^{q-1}\,dv_g=o(\mu_\alpha^2 r_\alpha^{-3}).$$
Furthermore,
$$\begin{array}{r l}
\displaystyle \int_{\Omega_\alpha}\nabla u_\alpha(Z_\alpha)f_\alpha u_\alpha^{q-1}\,dv_g=
&\displaystyle O\left(\int_{\partial\Omega_\alpha}u_\alpha^q\,d\sigma_g\right)\\ \\
&\displaystyle-\frac{1}{q}\int_{\Omega_\alpha}div_g(Z_\alpha)f_\alpha u_\alpha^q\,dv_g\\ \\
&\displaystyle -\frac{1}{q}\nabla f_\alpha(Z_\alpha)\int_{\Omega_\alpha}u_\alpha^q\,dv_g,
\end{array}
$$
which leads us to conclude the claim in (\ref{for nabla f}).
Note that $$div_g(Z_\alpha)=O\left(d_g(x_\alpha,x)^\eta\right)$$ and
$$\int_{\Omega_\alpha}u_\alpha^q\,dv_g\rightarrow\int_{R^n}\left(1+\frac{|x|^2}{R_0^2}\right)^{-5}\,dx<\infty.$$
\par Finally, we are in the position to remark that $\rho_\alpha=r_\alpha.$ Remember that $$\varphi(r)=\frac{1}{\omega_{n-1}r^{n-1}}\int_{\partial B_0(r)}\hat{U}=\left(\frac{R_0}{r}\right)^{n-2}+H(0)$$ and that $\left(r^{\frac{n-2}{2}}\varphi(r)\right)'(1)=0$, so if $r_\alpha<\rho_\alpha$, then $H(0)=R_0^{n-2}$, which contradicts (\ref{Hdelta}). Thus (\ref{prefinal profile}) implies (\ref{est final profile in r alpha}), and this wraps up the proof of the lemma.
\end{proof}
Moreover, $\rho_\alpha=r_\alpha$ means that $\rho_\alpha\to 0$ because $r_\alpha=O\left(\mu_\alpha^\frac{1}{2}\right)$ thanks to (\ref{size of r alpha}). As an important consequence, there do not exist any isolated bubbles. Otherwise, if a bubble were isolated, then we could choose a blow-up sequence with $0<\delta<\rho_\alpha$, contradicting the previous result. 
\subsection{Proof of the stability theorem}
We are now in the position to prove Theorem \ref{theo stability eq}. Let
$$\delta_\alpha:=\min_{1\leq i<j\leq N_\alpha}d_g(x_{i,\alpha}, x_{j,\alpha}).$$
\par For any $R>0$, let $1\leq M_{R,\alpha}$ be such that
$$d_g(x_{1,\alpha}, x_{i_\alpha,\alpha})\leq R\delta_\alpha\quad\text{ for }\quad i_\alpha\in\{1,\dots,M_{R,\alpha}\}, \text{ and}$$
$$d_g(x_{1,\alpha}, x_{j_\alpha,\alpha})> R\delta_\alpha \quad\text{ for }\quad j_\alpha\in\{M_{R,\alpha}+1,\dots, N_\alpha\}.$$
We consider the rescaled quantities
$$\check{u}_\alpha(x):=\delta_\alpha^{\frac{n-2}{2}}u_\alpha(\exp_{x_{1,\alpha}}(\delta_\alpha x))\quad\text{ and }\quad \check{g}_\alpha(x):=\left(\exp^*_{x_{1,\alpha}}g\right)(\delta_\alpha x)$$
and the coordinates $\check{x}_{i,\alpha}:=\delta_\alpha^{-1}\exp_{x_{1,\alpha}}^{-1}(x_{i,\alpha})$ in the exponential chart. It's obvious that $|\check{x}_{2,\alpha}|=1$ and $|\check{x}_{i,\alpha}|\geq 1$.
\par The following lemma is a direct consequence of Lemma \ref{lem: former lemma 3}.
\begin{lemma}\label{lem former lemma 9}
For all $R>0$, there exists $C_R>0$ such that the Harnack-type inequality
$$||\nabla \check{u}_\alpha||_{L^\infty(\Omega_R)}\leq C_R \sup_{\Omega_R}\check{u}_\alpha\leq C_R^2\inf_{\Omega_R}\check{u}_\alpha$$
holds, where $\Omega_R=B_0(R)\backslash\bigcup^{M_{2R,\alpha}}_{i=1}B_{\check{x}_{i,\alpha}}\left(\frac{1}{R}\right).$
\end{lemma}
Note that, for $1\leq i<j\leq M_{R,\alpha}$, $B_{x_{i,\alpha}}\left(\frac{\delta_\alpha}{4}\right)$ and $B_{x_{j,\alpha}}\left(\frac{\delta_\alpha}{4}\right)$ are disjoint, which is equivalent to saying that $B_{\check{x}_{i,\alpha}}\left(\frac{1}{4}\right)$ and $B_{\check{x}_{j,\alpha}}\left(\frac{1}{4}\right)$ are also disjoint. 
\par At this point, we are finally able to prove Theorem \ref{theo stability eq}, which we stated at the very beginning of this section. We define two possible types of concentration points, according to how $\check{u}_\alpha$ explodes. We prove that, within a cluster, we can only find one type or the other, but never both. Finally, we see that the existence of either type leads to contradictions, which implies that $\check{u}_\alpha$ admits no concentration points whatsoever. 
\begin{proof}[Proof of Theorem \ref{theo stability eq}:]
\par Consider the cluster around $(x_{1,\alpha})_\alpha$, for some $R>0$. There are two possible cases. The first type of concentration point corresponds to
\begin{equation}\label{first case concentration point}
\sup_{B_{\check{x}_{i,\alpha}}\left(\frac{1}{2}\right)}\left(\check{u}_\alpha(x)+\left|\frac{\nabla\check{u}_\alpha(x)}{\check{u}_\alpha(x)}\right|^{\frac{n-2}{2}}\right)=O(1).
\end{equation}
In this case, note that $(\check{u}_\alpha)_\alpha$ is uniformly bounded in $\mathcal{C}^1_{loc}$. Moreover, we find a lower bound, as by (\ref{est: 1 of former lemma 1}) from Lemma \ref{lem: former lemma 2},
$$|\check{x}_{i,\alpha}|^{\frac{n-2}{2}}\check{u}_\alpha(\check{x}_{i,\alpha})\geq 1.$$
There exists $\delta_i>0$ such that
$$
\inf_{B_{\check{x}_{i,\alpha}}(\delta_i)}\check{u}_\alpha\geq \frac{1}{2}|\check{x}_{i,\alpha}|^{1-\frac{n}{2}},
$$
which leads to the existence of $\delta_0>0$ where
\begin{equation}\label{lower bound}
\inf_{B_{\check{x}_{i,\alpha}}(\delta_0)}\check{u}_\alpha\geq \frac{1}{2}.
\end{equation}
The second type is defined by
\begin{equation}\label{second case concentration point}
\sup_{B_{\check{x}_{i,\alpha}}\left(\frac{1}{2}\right)}\left(\check{u}_\alpha(x)+\left|\frac{\nabla\check{u}_\alpha(x)}{\check{u}_\alpha(x)}\right|^{\frac{n-2}{2}}\right)\to\infty.
\end{equation}
In this case, either
\begin{equation}\label{nabla explodes}
\sup_{B_{\check{x}_{i,\alpha}}\left(\frac{1}{2}\right)}\check{u}_\alpha(x)\leq M\quad \text{ and }\quad \sup_{B_{\check{x}_{i,\alpha}}\left(\frac{1}{2}\right)}|\nabla\check{u}_\alpha(x)|\to\infty,
\end{equation}
or
\begin{equation}\label{l infty explodes}
\sup_{B_{\check{x}_{i,\alpha}}\left(\frac{1}{2}\right)}\check{u}_\alpha(x)\to\infty.
\end{equation}
We show (\ref{nabla explodes}) is not actually possible. Assume it holds true. Then there exist $(\check{x}_\alpha)_\alpha\subset (B_{\check{x}_{i,\alpha}}\left(\frac{1}{2}\right))_\alpha$ and $(\check{\nu}_\alpha)_\alpha$ such that 
$$
\check{\nu}_\alpha^{1-\frac{n}{2}}:= \check{u}_\alpha(\check{x}_\alpha)+\left|\frac{\nabla \check{u}_\alpha(\check{x}_\alpha)}{\check{u}_\alpha(\check{x}_\alpha)}\right|^{\frac{n-2}{2}}=\sup_{x\in B_{\check{x}_{i,\alpha}}\left(\frac{1}{2}\right)}\left(\check{u}_\alpha(x)+\left|\frac{\nabla \check{u}_\alpha(x)}{\check{u}_\alpha(x)}\right|^{\frac{n-2}{2}}\right)
$$
with
\begin{equation}
\check{\nu}_\alpha\to 0.
\end{equation}
We define the rescaled quantities
$$\check{v}_\alpha(x):=\check{\nu}_\alpha^{\frac{n-2}{2}}\check{u}_\alpha\left(\exp_{\check{x}_\alpha}(\check{\nu}_\alpha x)\right)\quad\text{and}\quad \check{h}_\alpha(x):=\left(\exp_{\check{x}_\alpha}^* \check{g}\right)(\check{\nu}_\alpha x)$$
respectively, defined in $\Omega_\alpha:=B_0\left(\frac{1}{2\check{\nu}_\alpha}\right)$. For any $R>0$ and $\alpha$ large enough so that $R<\frac{1}{2\check{\nu}_\alpha}$,
\begin{equation}
    \limsup_{\alpha\to\infty}\sup_{B_0(R)} \left(\check{v}_\alpha+\left|\frac{\nabla \check{v}_\alpha}{\check{v}_\alpha}\right|\right)=1.
\end{equation}
Thus
$$|\nabla \ln \check{v}_\alpha|\leq 1$$
and
\begin{equation}
    \check{v}_\alpha(0)e^{-x}\leq \check{v}_\alpha(x)\leq \check{v}_\alpha(0)e^{x}.
\end{equation}
Note that the metrics $\check{h}_\alpha\to\xi$ in $\mathcal{C}^2_{loc}$ as $\alpha\to\infty$. Assume that, up to a subsequence, $u_\alpha(\check{x}_\alpha)\to l<\infty$. We also deduce that $\check{v}_\alpha(0)\to 0.$ Let $\check{x}_0:=\lim_{\alpha\to\infty}\check{x}_\alpha$ and let us denote
$$\check{w}_\alpha(x):=\frac{\check{v}_\alpha(x)}{\check{v}_\alpha(0)}.$$
These functions are bounded from below,
\begin{equation}\label{est w again}
\check{w}_\alpha(x)\geq\frac{\varepsilon}{l}+o(1)>0.
\end{equation}
Moreover, 
$$\check{w}_\alpha(x)\leq e^{|x|}.$$
    By standard elliptic theory, we find that there exists $\check{w}:=\lim_{\alpha\to\infty}\check{w}_\alpha$ in $\mathcal{C}^1$ solving:
    $$\Delta \check{w}=-\frac{1}{l^{q+2}}\frac{\langle\nabla \check{w}, Y(\check{x}_0)\rangle^2}{\check{w}^{q+3}}.$$
    Note that 
    $$\Delta \check{w}^{-\alpha}\leq \alpha\frac{|\nabla \check{w}|^2}{\check{w}^{\alpha+2}}\left[\frac{||Y(\check{x}_0)||_{L^\infty}^2}{\varepsilon^{q+2}}-(\alpha+1)\right],$$
    so $\check{w}^{-\alpha}$ is subharmonic for $\alpha$ large, and so Lemma \ref{lem: A1} (see the Annex) implies that $\check{w}$ is constant, which in turn implies that $\check{U}=0$ and $\nabla \check{U}=0$, which is false (see proof of Lemma \ref{lem: former lemma 1}). Therefore, the second subcase cannot be true. This essentially means that when a concentration point is of the second type, then
    $$\sup_{B_{\check{x}_{i,\alpha}}\left(\frac{1}{2}\right)}\check{u}_\alpha(x)\to\infty$$
and so
$$\check{u}_\alpha(\check{x}_{i,\alpha})\to\infty.$$
\par Let us denote $\check{x}_i:=\lim_{\alpha\to\infty}\check{x}_{i,\alpha}$ up to a subsequence. According to Proposition \ref{prop},
\begin{equation}\label{sectype}
\check{u}_\alpha(\check{x}_{i,\alpha})\check{u}_\alpha(x)\mapsto\frac{\lambda_i}{|x-\check{x}_i|^{n-2}}+H_i(x)
\end{equation}
in $\mathcal{C}^1$ in $B_{\check{x}_i}\left(\frac{1}{2}\right)\backslash \{\check{x_i}\}$, with $\lambda_i>0$, where $H_i$ is a harmonic function in $B_{\check{x}_{i,\alpha}}\left(\frac{1}{2}\right)$, $H(\check{x}_i)=0$. 
\par Let $U$ be a connected open set of $\R^n$, $U_R\subset B_0(R+1)$, containing no other point of the cluster apart from $\check{x}_i$ and $\check{x}_j$. For any $0<r<\frac{1}{8}$, we set
$$V_{r,R}=U_R\backslash\left(\overline{B_{\check{x_i}}(r)\bigcup B_{\check{x_j}}(r)}\right).$$
\par For a fixed $x\in B_{\check{x}_i}\left(\frac{1}{4}\right)\backslash V_{r,R}$, (\ref{sectype}) implies that $\check{u}_\alpha(x)\to 0$ as $\alpha\to\infty$. It follows from Lemma \ref{lem former lemma 9} and (\ref{lower bound}) that all points of a cluster must be of the same type. 
\par Assuming all points in the cluster are of the first type, then
$$\check{u}_\alpha(0)+\left|\frac{\nabla\check{u}_\alpha(0)}{\check{u}_\alpha(0)}\right|^{\frac{n-2}{2}}=O(1),$$
then by standard elliptic theory there exists $\check{u}:=\lim_{\alpha\to\infty}\check{u}_\alpha$ in $\mathcal{C}^1(B_0(R))$, $R>0$. Repeating the reasoning of Lemma \ref{lem: former lemma 2} or Lemma \ref{lem: former lemma 4}, we know that  
$$\Delta_\xi \check{u}=f(x_1)\check{u}^{q-1}.$$
However, $\check{u}$ must have at least two separate maxima, at 0 and $\check{x}_2$, which leads to a contradiction by the classification result of Caffarelli, Gidas and Spruck \cite{CaffGidSpr}. 
\par Therefore $\check{u}_\alpha(\check{x}_{i,\alpha})\to\infty$, for any $i=\overline{1,M_{2R,\alpha}}$. Up to a subsequence
$$\frac{\check{u}_\alpha(\check{x}_{i,\alpha})}{\check{u}_\alpha(0)}\to\mu_i>0\quad\text{as}\quad\alpha\to\infty.$$
We fix $R>0$ and assume, without loss of generality, that $(M_{2R,\alpha})_\alpha$ is a constant denoted by $M_{2R}$. Using Lemma \ref{lem former lemma 9} and standard elliptic theory, we pass to a subsequence and get
$$\check{u}_\alpha(0)\check{u}_\alpha(x)\to\check{G}(x)$$
in $\mathcal{C}^1_{loc}\left(B_0(R)\backslash\{\check{x}_i\}_{i=\overline{1,N_{2R}}}\right)$ for $\alpha\to\infty$, with
$$
\begin{array}{r l}
\check{G}(x)&=\sum_{i=1}^{p}\frac{\lambda_i}{\mu_i|x-\check{x}_i|^{n-2}}+\check{H}(x)\\&=\frac{\lambda_1}{|x|^{n-2}}+\left(\sum_{i=2}^{p}\frac{\lambda_i}{\mu_i|x-\check{x}_i|^{n-2}}+\check{H}(x)\right)
\end{array}
$$
Here, $\check{H}$ is harmonic on $B_0(R)$, and $2\leq p\leq M_{2R}$ such that $|\check{x}_p|\leq R$ as $|\check{x}_{p+1}|>R$. If we apply Proposition \ref{prop} to the blow-up sequence $x_\alpha=x_{1,\alpha}$ with $\rho_\alpha=\frac{1}{16}d_\alpha$, we obtain
$$\hat{H}(0):=\sum_{i=2}^{p}\frac{\lambda_i}{\mu_i|\check{x}_i|^{n-2}}+\check{H}(0)=0$$
Since $\hat{H}(x)-\frac{\lambda_2}{\mu_2|x-\check{x}_{2}|}=\check{G}(x)-\frac{\lambda_1}{|x|^{n-2}}-\frac{\lambda_2}{\mu_2|x-\check{x}_2|^{n-2}}$ is harmonic in the ball $B_0(R)\backslash\{\check{x}_i\}_{i\in\overline{2,N_{2R}}}$ and $\check{G}\geq 0$, then as a consequence of the maximum principle, by considering a minimum on $\partial B_0(R),$ we see that
$$\hat{H}(0)\geq\frac{\lambda_2}{\mu_2}-\frac{\lambda_1}{R^{n-2}}-\frac{\lambda_2}{\mu_2(R-1)^{n-2}}.$$ 
Choosing $R>0$ large enough, we ensure that $\hat{H}>0,$ which contradicts Theorem \ref{lem: former lemma 7}. Consequently, $u_\alpha$ admits no concentration points and is therefore uniformly bounded in $\mathcal{C}^1$.
\end{proof}
\begin{lemma}\label{lem strong stability}
Assuming equation (\ref{EM}) associated to $\tilde{a}$ admits a supersolution and that $\Delta_g+h$ is coercive, then for any $0<T<\inf_M\tilde{a}$ and any equation with parameters in $\mathcal{E}_{\theta,T}$ (as in Theorem \ref{theo stability eq}), there exists a constant $C_{\theta,T}=C(n,\theta,T)>0$ such that, for any $||Y||_{L^\infty}\leq C_{\theta,T}$ and $||b||_{L^\infty}\leq C_{\theta,T}$, we may find a smallest real eigenvalue $\lambda_0>0$, where $\lambda_0$ is as in Lemma \ref{lem weak stability}.
\end{lemma}
\begin{proof}
Given any parameters $(f,a,b,c,d,h,Y)$ in $\mathcal{E}_{\theta,T}$ and additionally asking for $Y$ and $b$ to be sufficiently small in $L^\infty$ norm with respect to $\theta$ and $T$, we aim to prove that minimal solutions to the Lichnerowicz-type equation change continuously with their parameters. In order to do this, we study the sign of the smallest real eigenvalue associated to the linearisation around a minimal solution and show that it is positive by comparing it to the smallest real eigenvalue at $b=0$ and $Y=0$. Indeed, let $s>0$ a real number and $E_s$ the equation
\begin{equation}\label{eq theta}
\begin{array}{c}
E_s(u_s):=\Delta_g u_s+hu_s-fu_s^{q-1}-\frac{a}{u_s^{q+1}}+\frac{s b}{u_s}+c\langle\nabla u_s,s Y\rangle\left(\frac{d}{u_s^2}+\frac{1}{u_s^{q+2}}\right)\\+\frac{\langle\nabla u_s,s Y\rangle^2}{u_s^{q+3}}=0,
\end{array}
\end{equation}
with $u_s$ its minimal solution. Let $L_s$ be the linearisation of $E_s$ around $u_s$,
\begin{equation}
\begin{array}{c}\label{eq varphi}
\Delta_g\varphi_s+\Big[h-(q-1)fu_s^{q-2}+(q+1)\frac{a}{u_s^{q+2}}-\frac{s b}{u_s}-c\langle\nabla u_s,s Y\rangle\left(\frac{2d}{u_s^3}+\frac{q+2}{u_s^{q+3}}\right)\\-(q+3)\frac{\langle\nabla u_s,s Y\rangle^2}{u_s^{q+4}}\Big]\varphi_s+\langle\nabla\varphi_s,s Y \rangle\Big[c\left(\frac{d}{u_s^2}+\frac{1}{u_s^{q+2}}\right)+\frac{2\langle\nabla u_s,s Y\rangle}{u_s^{q+3}}\Big]=\lambda_s\varphi_s,
\end{array}
\end{equation}
with $\lambda_s\geq 0$ the smallest real eigenvalue, $\varphi_s>0$ the associated eigenfunction, normalised such that $||\varphi_s||_{L^2}=1.$ Note that the linear equations $L_s$ are stable, in the sense that $\varphi_s$ is \textit{a priori} uniformly bounded in $\mathcal{C}^1$. This follows from the fact that the $u_s$ is uniformly bounded. We may also suppose that $\lambda_s$ is uniformly bounded, because if $\lambda_s\to\infty$, then it is clear that $\lambda_s>0$. 
\par As Premoselli proved by way of a variational argument \cite{Pre14}, the equation $E_0$ is strictly stable, in the sense that its corresponding smallest real eigenvalue is positive. It uses the coerciveness of $\Delta_g+h$. We emphasize that his argument makes use of the fact that $E_0$ is symmetric, which is not the case for our more general equations. The strict stability implies continuity, \textit{i.e.} that $u_s\to u_0$, with $u_0$ the minimal value. Indeed, let $u_s\to \tilde{u}$ another solution of $E_0$. Clearly, $\tilde{u}>u_0$. Let $\tilde{u}_\delta=u_0+\delta\varphi_0$. Note that
$$
\begin{array}{r l}
E_s(\tilde{u}_\delta)&=E_0(\tilde{u}_\delta)+\frac{s b}{u_\delta}+c\langle\nabla u_\delta,s Y\rangle\left(\frac{d}{u_\delta^2}+\frac{1}{u_\delta^{q+2}}\right)+\frac{\langle\nabla u_\delta,s Y\rangle^2}{u_\delta^{q+3}}\\
&=E_0(u_0)+\lambda_0\delta\varphi_0+o(\delta)+\frac{s b}{u_\delta}+s c\langle\nabla u_\delta, Y\rangle\left(\frac{d}{u_\delta^2}+\frac{1}{u_\delta^{q+2}}\right)\\
&\quad +s^2\frac{\langle\nabla u_\delta, Y\rangle^2}{u_\delta^{q+3}}
\end{array}
$$
If we fix $\delta>0$ sufficiently small, the error terms $|o(\delta)|\leq \frac{\lambda_0\delta\varphi_0}{3}$ and $\tilde{u}_\delta<\tilde{u}\leq u_s$, $\forall s$. Then, by taking $s$ sufficiently close to $0$, we get that the rest of the terms are also smaller in absolute size than $\frac{\lambda_0\delta\varphi_0}{3}$. Consequently, $E_s(\tilde{u}_\delta)>0$, so $\tilde{u}_\delta$ is a supersolution of $E_s$ that is smaller than the minimal solution $u_s$. 
\par Since $u_s\to u_0$, we also get that $\lambda_s\to\lambda_0$, so for $s$ small, the first eigenvalue $\lambda_s>0$. We would like to obtain that there exists $s_{\theta,T}>0$ such that, for any $0\leq s<s_{\theta,T}$, the minimal eigenvalue corresponding to $L_s$ is positive, where $(a,b,c,d,f,h,Y)\in\mathcal{E}_{\theta,T}$. In other words, we attempt to set a size for $Y$ and $b$, depending on $\theta$ and $T$ (and $n$), such that the resulting equations are strictly stable.
\par First, there exists $\delta_{\theta,T}>0$ such that if $Y=0$, $b=0$ and the equation's parameters are found in $\mathcal{E}_{\theta,T}$, then $\lambda_0>\delta_{\theta,T}.$ We let $u_s=u_0+\varepsilon_s v_s$ such that $||v_s||_{L^2}=1$, $\varepsilon_s\in\R$. Note that $\varepsilon_s\to 0$ as $s\to 0$.
\par We begin by analyzing the difference in size between $\varepsilon_s$ and $s$, or equivalently between $||u_s-u_0||_{L^\infty}$ and $s$. Let
$$E_s=E_0+s M_s,$$
where 
$$M_s(u_s)=\frac{b}{u_s}+c\langle\nabla u_s, Y\rangle\left(\frac{d}{u_s^2}+\frac{1}{u_s^{q+2}}\right)+s\frac{\langle\nabla u_s, Y\rangle^2}{u_s^{q+3}}.$$
Recall that
\begin{equation}\label{Eq theta develop}
E_0(u_s)=-s M_s(u_s)
\end{equation}
where 
$$E_0(u_s)=L_0(u_s-u_0)+O(|u_s-u_0|^2).$$
Since $u_0$ is a solution of $E_0$ and the operator $L_0$ is coercive, with minimal eigenvalue $\lambda_0$, then by testing (\ref{Eq theta develop}) against $(u_s-u_0)$, we see that
$$\lambda_0\left(1+o(1)\right)||u_s-u_0||_{L^2}^2\leq -s \int_M M_s(u_s)(u_s-u_0)\leq s||M_s(u_s)||_{L^2}||u_s-u_0||_{L^2}.$$
The size of $M_s(u_s)$ is determined by a constant depending on $\theta$ and $T$. Therefore, we may write
$$(1+o(1))\varepsilon_s =(1+o(1))||u_s-u_0||_{L^2}\leq s \frac{C}{\lambda_0}$$
Finally, in order to compare $\lambda_s$ to $\lambda_0$, extract the terms of order $s$ from the quantity $\int_M \varphi_0L_s(\varphi_s)-\varphi_s L_0(\varphi_0)$,
\begin{equation}\label{eq L theta order varepsilon theta}
\begin{array}{c}
-\int_M\frac{s b}{u_0}\varphi_0\varphi_s-\int_Mc\langle\nabla u_0,s Y\rangle\left(\frac{2d}{u_0^3}+\frac{q+3}{u_0^{q+3}}\right)\varphi_s\varphi_0-\int_M(q+3)\frac{\langle\nabla u_s,s Y\rangle^2}{u_s^{q+4}}\varphi_s\varphi_0\\
\varepsilon_s\int_M\left[(q-1)(q-2)fu_0^{q-3}-(q+1)(q+2)\frac{a}{u_0^{q+3}}\right]v_s \varphi_s\varphi_0\\
+O(s^2)=(\lambda_\theta-\lambda_0)\int_M\varphi_s\varphi_0,
\end{array}
\end{equation}
so there exists a constant $C$ depending on $\theta$ and $T$ such that
\begin{equation}
|\lambda_s-\lambda_0|\leq s C\left(\int_M\varphi_s\varphi_0\right)^{-1}.
\end{equation}
As $\varphi_s=\varphi_0+o(1)$ and the $L^2$ norm of $\varphi_0$ is $1$, we may choose $s$ small enough so that $|\lambda_s-\lambda_0|\geq \frac{\delta_{\theta,T}}{2}$, and thus $\lambda_s>0$.
\end{proof}
\section{Existence of solutions to the system}\label{theo: system}
\subsection{The proof of the main theorem}
The following is a useful estimate we can find in \cite{IseMur}; it plays a crucial role in ensuring the necessary compacity of the sequence $W_\alpha$ in the main theorem.
\begin{prop}\label{prop isemur} Let $(M,g)$ be a closed Riemannian manifold of dimension $n\geq 3$ such that $g$ has no conformal Killing fields. Let $X$ be a smooth vector field in $M$. Then there exists a unique solution $W$ of
$$\Delta_{g, conf}W=X.$$
Also, for $0<\gamma<1$, there exists a constant $C_0>0$ that depends only on $n$ and $g$ such that
$$||W||_{\mathcal{C}^{1,\gamma}}\leq C_0||X||_\infty.$$
\end{prop}
\begin{rem}
As a consequence, there exists a constant $C_1=C_1(n,g)$ such that
\begin{equation}\label{lxest}
||\mathcal{L}_g W||_{\mathcal{C}^{0,\gamma}}\leq C_1||X||_\infty
\end{equation}
\end{rem}
Let $(M,g)$ be a closed Riemannian manifold of dimension $n\in\{3,4,5\}$ such that $g$ has no conformal Killing fields. Let $b$, $c$, $d$, $f$, $h$, $\rho_1$, $\rho_2$, $\rho_3$ be smooth functions on $M$ and let $Y$ and $\Psi$ be smooth vector fields defined on $M$. Let $0<\gamma<1$. 
\par Assume that $\Delta_g+h$ is coercive. Assume that $f>0$, $\rho_1>0$ and $|\nabla \rho_3|<(2C_1)^{-1}$, where $C_1$ is defined in (\ref{lxest}). 
\par Consider the coupled system
\begin{equation}\label{system in proof}
\begin{cases}
\displaystyle \Delta_g u+hu
    &\displaystyle= fu^{q-1}+\frac{\rho_1+|\Psi+\rho_2\mathcal{L}_gW|^2_g}{u^{q+1}}\\ 
    &\displaystyle -\frac{b}{u}-c\langle\nabla u,Y\rangle \left(\frac{d}{u^2}+\frac{1}{u^{q+2}}\right)-\frac{\langle\nabla u,Y\rangle^2}{u^{q+3}}\\
    \displaystyle div_{g}\left(\rho_3\mathcal{L}_g W\right) 
    &\displaystyle =\mathcal{R}(u),
\end{cases}
\end{equation}
where $\mathcal{R}$ is an operator verifying
\begin{equation}\label{CR}
\mathcal{R}(u)\leq C_\mathcal{R}\left(1+\frac{||u||_{\mathcal{C}^2}^2}{(\inf_Mu)^2}\right)
\end{equation}
for a constant $C_\mathcal{R}>0$. 
\par We fix 
\begin{equation}
\label{theta}
\theta=\min(\inf_M\rho_1,\inf_M f),
\end{equation}
and 
\begin{equation}\label{T2}
T=\max(||f||_{\mathcal{C}^{1,\eta}}, ||\rho_1||_{\mathcal{C}^{0,\gamma}}, ||c||_{\mathcal{C}^{0,\gamma}}, ||d||_{\mathcal{C}^{0,\gamma}}, ||h||_{\mathcal{C}^{0,\gamma}}).
\end{equation}
Let 
\begin{equation}\label{bound m}
M=\ln S_{\theta,2T},
\end{equation}
with $S_{\theta,2T}$ a constant as in Theorem \ref{theo stability eq}.
The following theorem is the main result of the present paper.
\begin{theo}\label{thm plus 1}
Assume there exists a smooth positive function $\tilde{a}$ for which
\begin{equation}\label{licheq}
    \Delta_g \tilde{u}+ h\tilde{u}=f\tilde{u}^{q-1}+\frac{\tilde{a}}{\tilde{u}^{q+1}}
\end{equation}
admits a positive supersolution $\tilde{u}$. Assume that $\rho_1<\tilde{a}$ and let $\omega=\inf_M(\tilde{a}-\rho_1)$. Then there exists
\begin{equation}
\delta=\delta(\omega,\theta,T)>0
\end{equation}
such that if 
\begin{equation}
||b||_{\mathcal{C}^{0,\gamma}}+ ||Y||_{\mathcal{C}^{0,\gamma}}+ ||\Psi||_{\mathcal{C}^{0,\gamma}}+||\rho_2||_{\mathcal{C}^{0,\gamma}}+ C_\mathcal{R}\leq \delta,
\end{equation}
the system (\ref{system in proof}) admits a solution $(u,W)$, with $u$ a smooth positive function and $W$ a smooth vector field.
\end{theo} 
\begin{rem}
We can use a result by Hebey, Pacard and Pollack (\cite{HebPacPol08}, Corollary 3.1) in order to ensure the existence of a supersolution $\tilde{u}$. There exists a constant $C=C(n,h)$, $C>0$ such that if $\tilde{a}$ is a smooth positive function verifying
\begin{equation}
||\tilde{a}||_{L^1(M)}\leq C(n,h)\left(\max_M |f|\right)^{1-n},
\end{equation}
then (\ref{licheq}) accepts a smooth positive solution. 
\end{rem}
\begin{proof}[Proof of Theorem \ref{thm plus 1}:]
The proof of the theorem consists of a fixed-point argument. Formally, we define the operator
$$\Phi:\varphi\to\ln u\left(\mathcal{L}_gW(e^\varphi)\right),$$
where $W(e^\varphi)$ solves the second equation of (\ref{system in proof}) for a fixed $u=e^\varphi$ and where $u\left(\mathcal{L}_gW(e^\varphi)\right)$  is the solution of the scalar equation of (\ref{system in proof}) constructed in Section 2 for a fixed $W(e^\varphi)$. In order to apply Schauder's fixed point theorem, we show that $\Phi:B_M\to B_M$, $B_M:=\{\varphi\in\mathcal{C}^2(M), ||\varphi||_{\mathcal{C}^2}\leq M\}$, where $M$ is defined as in (\ref{bound m}), and that $\Phi$ is continuous and compact.
\par We first want to prove that
\begin{equation}\label{smaller than a tilde}
\rho_1+|\Psi+\rho_2\mathcal{L}_gW(e^\varphi)|^2_g<\tilde{a}
\end{equation}
to ensure that $\Phi(\varphi)$ is well defined, with $\tilde{u}$ from (\ref{licheq}) a supersolution.
By (\ref{lxest}), we have
\begin{equation}\label{l c gamma}
||\mathcal{L}_g W(e^\varphi)||_{\mathcal{C}^{0,\gamma}}\leq C_1\left(||\nabla\rho_3||_{L^\infty}||\mathcal{L}_g W(e^\varphi)||_{L^\infty}+||\mathcal{R}(e^\varphi)||_{L^\infty}\right)
\end{equation}
and thanks to (\ref{CR}) we see that
\begin{equation}
\begin{array}{c}
\rho_1+|\Psi+\rho_2\mathcal{L}_gW(e^\varphi)|^2_g\leq \rho_1+ 2||\Psi||_{L^\infty}^2\\
+2\left(\frac{C_1 C_{\mathcal{R}}||\rho_2||_{L^\infty}}{1-C_1||\nabla \rho_3||_{L^\infty}}\right)^2 \left(1+\frac{M^2e^{2M}}{e^{2\varepsilon}}\right)^2,
\end{array}
\end{equation}
where $\varepsilon$ is the lower bound of any solution corresponding to $\mathcal{E}_{\theta,2T}$ from Theorem \ref{theo stability eq}. There exists
\begin{equation}\label{delta1}
\delta_1=\delta_1(\omega,\theta,T)>0
\end{equation}
such that if
\begin{equation}
||\Psi||_{\mathcal{C}^{0,\gamma}}+||\rho_2||_{\mathcal{C}^{0,\gamma}}+C_{\mathcal{R}}\leq \delta_1,
\end{equation}
then (\ref{smaller than a tilde}) holds. 
\par In order to use the \textit{a priori} estimate of Section 3 to see that $\Phi:B_M\to B_M$, we need to prove that
\begin{equation}\label{ee}
\theta\leq \rho_1+|\Psi+\rho_2\mathcal{L}_gW(e^\varphi)|^2_g \quad\text{ and }\quad  ||\rho_1+|\Psi+\rho_2\mathcal{L}_gW(e^\varphi)|^2_g||_{\mathcal{C}^{0,\gamma}}\leq 2T.
\end{equation}
\par From (\ref{theta}) we deduce that
\begin{equation}\label{ee1}
\theta\leq \rho_1+|\Psi+\rho_2\mathcal{L}_gW(e^\varphi)|^2_g.
\end{equation}
and thanks to (\ref{l c gamma}) we see that
\begin{equation}
\begin{array}{c}
||\rho_1+|\Psi+\rho_2\mathcal{L}_gW(e^\varphi)|^2_g||_{\mathcal{C}^{0,\gamma}}\leq ||\rho_1||_{\mathcal{C}^{0,\gamma}}+ 2||\Psi||_{\mathcal{C}^{0,\gamma}}^2\\
+2\left(\frac{C_1 C_{\mathcal{R}}||\rho_2||_{\mathcal{C}^{0,\gamma}}}{1-C_1||\nabla \rho_3||_{L^\infty}}\right)^2 \left(1+\frac{M^2e^{2M}}{e^{2\varepsilon}}\right)^2.
\end{array}
\end{equation}
There exists 
\begin{equation}\label{delta2}
\delta_2=\delta_2(\omega,\theta, T)>0
\end{equation}
such that if
\begin{equation}
||\Psi||_{\mathcal{C}^{0,\gamma}}+||\rho_2||_{\mathcal{C}^{0,\gamma}}+C_{\mathcal{R}}\leq \delta_2,
\end{equation}
then
\begin{equation}\label{ee2}
||\rho_1+|\Psi+\rho_2\mathcal{L}_gW(e^\varphi)|^2_g||_{\mathcal{C}^{0,\gamma}}\leq 2T.
\end{equation}
Thanks to (\ref{ee1}) and (\ref{ee2}), the \textit{a priori} estimates in Section 3 imply that
$$||u\left(\mathcal{L}_gW(e^\varphi)\right)||_{\mathcal{C}^2}\leq S_{\theta,2T},$$
so 
\begin{equation*}
\Phi(\varphi)\leq M,
\end{equation*}
where $M$ is as in (\ref{bound m}). We have thus proved that $\Phi$ is well-defined and that $\Phi:B_M\to B_M$.
\par In order to show that $\Phi$ is continuous, we want to check that it holds true for $a\mapsto u(a)$, where $u(a)$ is the minimal solution constructed in Section 2. For all $a<\tilde{a}$, we've established monotony, which ensures that the minimal solutions exist. For $t>0$ small, let us denote by $u_t$ the solutions corresponding to $a(1+t)<\tilde{a}$. Let $u_0$ be the limit of $u_t$ as $t\to 0$; it is also a solution of the Lichnerowicz-type equation associated to $a$. If $u_0\not=u,$ then $u<u_0$. According to Section 3, there exists $C_{\theta,2T}>0$ such that
\begin{equation}
||b||_{\mathcal{C}^{0,\gamma}}+||Y||_{\mathcal{C}^{0,\gamma}}\leq C_{\theta,2T}
\end{equation}
implies that $u$ is strictly stable. We ask that
\begin{equation}
\delta\leq\min\left( \delta_1, \delta_2, C_{\theta,2T}\right)
\end{equation}
where $\delta_1$ is defined in (\ref{delta1}) and $\delta_2$ is defined in (\ref{delta2}). We choose $\mu>0$ small enough such that $u< \hat{u}_\mu< u_0$, where $\hat{u}_\mu:=u+\mu\psi$, $\psi$ a positive eigenfunction at $u$ corresponding to the smallest real eigenvalue.  But $\hat{u}_\mu$ is a supersolution for $a(1+\epsilon)$, $\epsilon>0$ small, which contradicts the monotonicity. Therefore, $\Phi$ is continuous.
\par Lastly, $B_M$ being a closed convex set in $\mathcal{C}^2$, it remains to show that $\Phi(B_M)$ is compact to conclude. From the previous discussion, $\Phi(B_M)\subset B_M$, and is thus bounded in $\mathcal{C}^2$. By standard elliptic theory, we conclude the proof of Theorem \ref{thm 1}.
\end{proof}
\subsection{The case of a metric with conformal Killing fields}
Let us consider the case of a metric $g$ with non-trivial conformal Killing fields associated to it. For $\tilde{V}$ a representative of the drift, the equation
\begin{equation}
div_{\bar{g}}\left(\frac{\tilde{N}}{2}\mathcal{L}_{g}W\right)=\frac{n-1}{n}u^q\mathbf{d}\left(u^{-2q}\tilde{N}div_{g}(u^q\tilde{V})\right)+\pi\nabla\psi
\end{equation} 
admits a solution $W$ if and only if
\begin{equation}
\frac{n-1}{n}\int_Mu^{-2q}\tilde{N}div_{g}(u^q\tilde{V})div_{g}(u^q P)=\int_M\langle\pi\nabla\psi,P\rangle
\end{equation}
for all $P$ conformal Killing fields. Moreover, the solution $W$ is unique up to the addition of a conformal Killing field. Note that the drift is defined modulo conformal Killing fields, so $\tilde{V}$ and $\tilde{V}+P$ are representatives of the same drift for all $P$ conformal Killing fields. We claim that given a vector field $\tilde{V}$ there exists a conformal Killing field $\tilde{Q}$ which is unique up to a true Killing field and such that
\begin{equation}\label{Q eq}
\frac{n-1}{n}\int_M u^{-2q}\tilde{N}div_{g}\left(u^q(\tilde{V}+\tilde{Q})\right)div_{g} (u^qP)=\int_M\langle\pi\nabla\psi,P\rangle.
\end{equation}
By analyzing the homogeneous operator associated to the equation above,
\begin{equation}
\int_Mu^{-2q}\tilde{N}div_{g}\left(u^q(\tilde{V}+\tilde{Q}')\right)div_{g}(u^q P)=0,
\end{equation}
we check that it is positive definite, thus invertible. Consider the functional 
\begin{equation}
F(P)=\int_Mu^{-2q}\tilde{N}div_{g}\left(u^q(\tilde{V}+P)\right)^2\,dv_{g}
\end{equation}
on the finite-dimensional space of conformal Killing fields and note that $\tilde{Q}'$ is stationary for $F$. Since $F$ is quadratic and non-negative definite, stationary points are associated to minimizers. If $\bar{g}$ does not admit any nontrivial true Killing fields, then every conformal Killing field $P$ satisfies $div P\not = 0$ and the quadratic term of $F$ is positive definite. On the other hand, if $g$ admits proper Killing fields, then $F$ descends to a functional on the quotient space and its quadratic order term is again positive definite. So the minimum of $F$ is unique up to a true Killing field.
\par The conformal system proposed by Maxwell \cite{Max14} becomes in this framework
\begin{equation}\label{syst of Maxwell - with CKF}
	\begin{cases}
	\Delta_g u+\frac{n-2}{4(n-1)}(R(g)-|\nabla\psi|_g^2)u -\frac{(n-2)}{4(n-1)}\frac{|U+\mathcal{L}_g W|^2+\pi^2}{u^{q+1}}\\ 
	\displaystyle\quad-\frac{n-2}{4(n-1)}\left[2V(\psi)-\frac{n-1}{n}\left(\tau^*+\frac{\tilde{N}div_g\left(u^q(\tilde{V}+\tilde{Q})\right)}{u^{2q}}\right)^2\right]u^{q-1}=0\\ \\
	 div_g\left(\frac{\tilde{N}}{2}\mathcal{L}_g W\right)=\frac{n-1}{n}u^q \mathbf{d}	\left(\frac{\tilde{N}div_g(u^q\left(\tilde{V}+\tilde{Q}\right))}{2u^{2q}}\right)+\pi\nabla\psi,
	\end{cases}
	\end{equation}
whose solution $(u,W,\tilde{Q})$ is a smooth positive function $u$, a smooth vector field $W$, defined up to a conformal Killing field, and $\tilde{Q}$ a conformal Killing field defined up to a true Killing field.
\par The existence of solutions to (\ref{syst of Maxwell - with CKF}) follows from Theorem \ref{thm 1} and is similar to Corollary \ref{physical result}, with slight modifications. Here,
\begin{equation}
\begin{array}{c}
\rho_1=\frac{n-2}{4(n-1)}\left[\pi-\frac{n-1}{n}\tilde{N}^2 div_g(\tilde{Q}+\tilde{V})\right],\quad b=-\tau^*\tilde{N}div_g(\tilde{V}+\tilde{Q}),\\
Y=\sqrt{\frac{n}{n-2}}\tilde{N}(\tilde{V}+\tilde{Q})
\end{array}
\end{equation}
and
\begin{equation}
C_\mathcal{R}=C_\mathcal{R}(||\tilde{Q}||_{\mathcal{C}^2}).
\end{equation}
Moreover, we define $\theta$ and $T$ as in (\ref{th main}) and (\ref{T main}) respectively, but without the dependency on $\rho_1=\rho_1(\tilde{Q})$, \textit{i.e.}
\begin{equation}
\theta=\min(\inf_M f),
\end{equation}
and 
\begin{equation}
T=\max(||f||_{\mathcal{C}^{1,\eta}}, ||c||_{\mathcal{C}^{0,\gamma}}, ||d||_{\mathcal{C}^{0,\gamma}}, ||h||_{\mathcal{C}^{0,\gamma}}).
\end{equation}
First of all, the stability of the first equation still holds, as in Lemma \ref{lem weak stability} and Lemma \ref{lem strong stability}. In order to apply the last theorem, we need to check that: $\rho_1(\tilde{Q})>\theta$, $||\rho_1(\tilde{Q})||_{\mathcal{C}^{0,\gamma}}<2T$, $||b||_{\mathcal{C}^{0,\gamma}}\leq C_{\theta,2T}$ and $||Y||_{\mathcal{C}^{0,\gamma}}\leq C_{\theta, 2T}$. This translates to
\begin{equation}
div_g\tilde{Q}<\left(\frac{n-2}{4(n-1)}\pi-\theta\right)\frac{n}{n-1}\tilde{N}^{-2}-div_g\tilde{V},
\end{equation}
\begin{equation}
\frac{n-2}{4(n-1)}||\pi||_{\mathcal{C}^{0,\gamma}}+\frac{n-1}{n}||\tilde{N}div_g(\tilde{V}+\tilde{Q})||_{\mathcal{C}^{0,\gamma}}\leq 2T,
\end{equation}
\begin{equation}
||\tau^*\tilde{N}div_g(\tilde{V}+\tilde{Q})||_{\mathcal{C}^{0,\gamma}}\leq C_{\theta,2T},
\end{equation}
and
\begin{equation}
\sqrt{\frac{n}{n-2}}||\tilde{N}(\tilde{V}+\tilde{Q})||\leq C_{\theta,2T}.
\end{equation}
We find bounds on $\tilde{Q}$ depending on $\pi$, $\psi$, $\tilde{N}$, $\tilde{V}$ from (\ref{Q eq}), thereby proving the necessary compactness. Finally, the continuity $(a,b,Y)\to u_{a,b,Y}$ doesn't pose any problem, and the proof mirrors our previous argument for the continuty of $a\to u(a).$ This shows the existence of solutions $(u,W,\tilde{Q})$. 
\section{Annex}
We used the following result repeatedly throughout the paper.
\begin{lemma}\label{lem: A1} Let $u$ be a bounded subharmonic function defined on $\R^n$. If there exists $0<\varepsilon\leq u$ which bounds $u$ from below and $\alpha>0$ such that $u^{-\alpha}$ is a subharmonic function, then $u$ is a constant.
\end{lemma}
\begin{proof}[Proof of Lemma \ref{lem: A1}:]
Let us denote
$$\bar{u}_{x}(R):=\frac{1}{\omega_{n-1}R^{n-1}}\int_{\partial B_{x}(R)}u(y)\,dy$$
the average of a smooth function $u$ over the sphere $\partial B_{x}(R)$. We will sometimes use the simplified notation $\bar{u}(R)$. Recall that, given any subharmonic function $u$, $x\in\R^n$ and for any two radii $R\leq\tilde{R}$, then
\begin{equation}\label{decreasing averages}
\bar{u}_{x}(R)\leq \bar{u}_x(\tilde{R}).
\end{equation}
This follows from
$$r^{n-1}\bar{u}'(r)=\frac{1}{\omega_{n-1}}\int_{\partial B_x{(r)}}\partial_\nu u(y)\,dy=-\frac{1}{\omega_{n-1}}\int_{B_x{(r)}}\Delta u(y)\,dy\geq 0$$
where $r>0$ and $\nu$ is the exterior normal.
\par Note that $u^{-\alpha}\leq \varepsilon^{-\alpha}$ implies that the average of $u^{-\alpha}$ on arbitrary subsets is uniformly bounded. Let us fix $x\in\R^n$. Since $u^{-\alpha}$ is bounded, there exists a constant $M>0$ and a sequence of radii $R_i\to\infty$ as $i\to\infty$ such that
\begin{equation}\label{def M}
M^{-\alpha}:=\lim_{i\to\infty}\overline{u^{-\alpha}}_{x}(R_i).
\end{equation}
In fact, because the averages are decreasing (\ref{decreasing averages}), any sequence $R\to\infty$ around any point in $\R^n$ leads to the same limit $M$, since one may always find a subsequence of $R_i$ such that $B_x(R_i)$ includes the new sequence.  
\par As $u^{-\alpha}$ is subharmonic,
$$u^{-\alpha}(x)\leq \overline{u^{-\alpha}}_x(R)$$
and therefore $u^{-\alpha}(x)\leq M^{-\alpha}$, or equivalently 
\begin{equation}\label{M smaller than u x}
M\leq u(x).
\end{equation}
\par For $z\in \R^n$, let $R:=|z-x|$ and $\tilde{R}>R$. By Green's representation theorem, we get
\begin{equation}\label{est: from the annex}
\begin{array}{r l}
\displaystyle u(z) &\displaystyle \leq \int_{\partial B_x(\tilde{R})}u(y)\frac{\tilde{R}^2-R^2}{\omega_{n-1}\tilde{R}|z-y|^n}\, dy\\ \\
&\displaystyle\leq \frac{(\tilde{R}+R)\tilde{R}^{n-2}}{(\tilde{R}-R)^{n-1}}\overline{u}_x(\tilde{R}).
\end{array}
\end{equation}
For $\delta>0,$ we denote
$$\Omega_{\delta,R}:=\{z\in\partial B_x(R), u(z)\geq M+\delta\}$$
a subset of $\partial B_x(R)$ and let
$$\theta_{\delta, R}:=\frac{|\Omega_{\delta, R}|}{|\partial B_x(R)|}\in[0,1]$$
be the corresponding relative size of its volume. Note that $\theta_{\delta,R}\to 0$ as $R\to\infty$. Otherwise, if there exists $\varepsilon\in(0,1]$ such that
$$\limsup_{R\to\infty}\frac{|\{z\in\partial B_x(R), u(z)\geq M+\delta\}|}{|\partial B_x(R)|}=\varepsilon$$
then
$$\limsup_{R\to\infty}\overline{u^{-\alpha}}_x(R) \leq \varepsilon(M+\delta)^{-\alpha}+(1-\varepsilon)M^{-\alpha}<M^{-\alpha}$$
which contradicts our definition (\ref{def M}) of $M$.
\par By choosing $R$ large, $\theta_{\delta,R}\leq \delta$. Let
$$\lambda_{\delta,i}:=\bar{u}_x(2^iR)$$
Note that, by (\ref{est: from the annex}), $\lambda_{\delta,i}\leq 3\times 2^{n-2}\lambda_{\delta, i+1}.$ Since
$$u(x)\leq\lambda_{\delta,i}\leq(M+\delta)(1-\theta_{\delta,2^iR})+\lambda_{\delta,i+1}\times\theta_{\delta,2^iR}$$
then, by induction,
$$u(x)\leq (M+\delta)\frac{1-\delta^l}{1-\delta}+\lambda_l\delta^l$$
for all $l\in \N$. As we take $l\to\infty$,
$$u(x)\leq(M+\delta)\frac{1}{1-\delta}$$
for any $\delta>0$, and therefore $u(x)\leq M.$ By (\ref{M smaller than u x}), $u(x)\equiv M.$
\par We may apply the same argument to any other $\tilde{x}\in\R^n$ and obtain the same value $u(\tilde{x})=M$. Indeed, assuming that 
$$\tilde{M}^{-\alpha}:=\lim_{\tilde{R}\to\infty}\overline{u^{-\alpha}}_{\tilde{x}}(\tilde{R})$$
so that $\tilde{M}^{-\alpha}\geq M^{-\alpha}$, then for $\tilde{R}$ large, $\overline{u^{-\alpha}}_{\tilde{x}}(\tilde{R})\geq M^{-\alpha}$. But, at the same time, given any fixed $\tilde{R}$, then for $R$ sufficiently large, by (\ref{est: from the annex}), $\overline{u^{-\alpha}}_{\tilde{x}}(\tilde{R})\leq \overline{u^{-\alpha}}_{x}(R)$. Thus we obtain that $u\equiv M$ in $\R^n$.
\end{proof}

\end{document}